\documentclass[11pt]{article}
\usepackage{amsmath,amsfonts,latexsym, xspace,amsthm,url,graphicx}
\usepackage[letterpaper,left=1in,right=1in,top=1in,bottom=1in]{geometry}

\newcommand{\E}[1]{\textbf{E} \left[#1\right]}

\usepackage{color}

\newcommand{\set}[1]{\left\{#1\right\}}

\def\lbar{\overline{\lambda}}
\def\lhat{\widehat{\lambda}}

\def\cB{{\cal B}}

\def\cD{{\cal D}}
\def\cE{{\cal E}}

\def\F{\Phi}

\def\a{\alpha}

\def\d{\delta}
\def\D{\Delta}
\def\e{\varepsilon}
\def\f{\phi}

\def\G{\Gamma}
\def\k{\kappa}

\def\z{\zeta}

\def\l{\lambda}
\def\m{\mu}
\def\n{\nu}
\def\p{\pi}

\def\s{\sigma}
\def\S{\Sigma}
\def\t{\tau}
\def\om{\omega}

\def\whp{w.h.p.}
\def\wvhp{w.v.h.p.}
\newcommand\Prob[1]{{\mbox{Pr}\left\{#1\right\}}}

\newcommand{\beq}[2]{\begin{equation}\label{#1}#2\end{equation}}

\newtheorem{lemma}{Lemma}
\newtheorem{theorem}{Theorem}

\newcommand{\brac}[1]{\left( #1\right)}
\newcommand{\bfrac}[2]{\brac{\frac{#1}{#2}}}
\newtheorem{conjecture}{Conjecture}

\newcommand{\rdup}[1]{\lceil #1 \rceil}
\newcommand{\rdown}[1]{{\mbox{$ \lfloor #1 \rfloor $}}}

\title{On edge disjoint spanning trees in a randomly weighted complete graph}
\author{Alan Frieze\thanks{Research supported in part by NSF Grant  DMS1362785. Email: alan@random.math.cmu.edu}\ \ \  Tony Johansson\thanks{Research supported in part by NSF Grant  DMS1362785. Email: alan@random.math.cmu.edu}\\ Department of Mathematical Sciences\\Carnegie Mellon University\\Pittsburgh PA 15213\\U.S.A.}

\begin{document}
\maketitle

\begin{abstract}
Assume that the edges of the complete graph $K_n$ are given independent uniform $[0,1]$ edges weights. We consider the expected minimum total weight $\mu_k$ of $k\geq 2$ edge disjoint spanning trees. When $k$ is large we show that $\mu_k\approx k^2$. Most of the paper is concerned with the case $k=2$. We show that $\m_2$ tends to an explicitly defined constant and that $\mu_2\approx 4.1704288\ldots$. 
\end{abstract}
\section{Introduction}
This paper can be considered to be a contribution to the following general problem. We are given a combinatorial optimization problem where the weights of variables are random. What can be said about the random variable equal to the minimum objective value in this model. The most studied examples of this problem are those of (i) Minimum Spanning Trees e.g. Frieze \cite{f85}, (ii) Shortest Paths e.g. Janson \cite{J99}, (iii) Minimum Cost Assignment e.g. Aldous \cite{A92}, \cite{A01}, Linusson and W\"astlund  \cite{LW04} and Nair,  Prabhakar and Sharma \cite{NPS05}, W\"astlund \cite{W09} and (iv) the Travelling Salesperson Problem e.g. Karp \cite{K1}, Frieze \cite{f04} and W\"astlund \cite{W10}.

The minimum spanning tree problem is a special case of the problem of finding a minimum weight basis in an element weighted matroid. Extending the result of \cite{f85} has proved to be difficult for other matroids. We are aware of a general result due to Kordecki and Lyczkowska-Han\'ckowiak \cite{KL} that expresses the expected minimum value of an integral using the Tutte Polynomial. The formulae obtained, although exact, are somewhat difficult to penetrate. In this paper we consider the union of $k$ cycle matroids. We have a fairly simple analysis for $k\to\infty$ and a rather difficult analysis for $k=2$. 

Given a connected simple graph $G=(V,E)$ with edge lengths ${\bf x} = (x_e:e \in E)$ and a positive integer $k$,
let $\mbox{mst}_k(G,{\bf x})$ denote the minimum length of $k$ edge disjoint spanning trees of $G$. ($\mbox{mst}_k(G)=\infty$ if such trees do not exist.) When ${\bf X} = (X_e:e \in E)$ is a family of independent random variables, each uniformly distributed on the interval $[0,1]$, denote the expected value $\E{\mbox{mst}_k(G,{\bf X})}$ by $\mbox{mst}_k(G)$. 

As previously mentioned, the case $k=1$ has been the subject of some attention. When $G$ is the complete graph $K_n$, Frieze \cite{f85} proved that
$$\lim_{n\to\infty}\mbox{mst}_1(K_n)=\z(3)=\sum_{k=1}^\infty\frac{1}{k^3}.$$
Generalisations and refinements of this result were subsequently given in Steele \cite{st}, Frieze and McDiarmid \cite{fm89}, Janson \cite{Ja}, Penrose \cite{Pen}, Beveridge, Frieze and McDiarmid \cite{BFM}, Frieze, Ruszinko and Thoma \cite{frt} and most recently in Cooper, Frieze, Ince, Janson and Spencer \cite{CFIJS}.

In this paper we discuss the case $k\geq 2$ when $G=K_n$ and define
$$\m_k^*=\liminf_{n\to\infty}\mbox{mst}_k(K_n)\text{ and }\m_k^{**}=\limsup_{n\to\infty}\mbox{mst}_k(K_n).$$
\begin{conjecture}\label{conj1}
$\m_k^*=\m_k^{**}$ i.e. $\lim_{n\to\infty}\mbox{mst}_k(K_n)$ exists.
\end{conjecture}
\begin{theorem}\label{th1a}
$$\lim_{k\to \infty}\frac{\m_k^*}{k^2}=\lim_{k\to \infty}\frac{\m_k^{**}}{k^2}=1.$$
\end{theorem}
\begin{theorem}\label{th1b}
With $f_k$ and $c_2'\approx 3.59$ and $\l_2' \approx 2.688$ as defined in \eqref{fkdef}, \eqref{ckdash}, \eqref{l2pdef},
\begin{multline*}
\m_2
=2c_2'-\frac{(c_2')^2}{4}+  \int_{\l=\l_2'}^\infty \left(2 - \frac{\l e^\l}{2f_2(\l)} + \frac{\l f_2(\l)}{2e^\l} - 2\frac{f_3(\l)}{e^\l} \right) \left(\frac{e^\l}{f_2(\l)} + \frac{\l e^\l}{f_2(\l)} - \frac{\l e^\l f_1(\l)}{f_2(\l)^2}\right) d\l\\
= 4.17042881\dots
\end{multline*}
\end{theorem}
There appears to be no clear connetion between $\mu_2$ and the $\z$ function.

Before proceeding to the proofs of Theorems \ref{th1a} and \ref{th1b} we note some properties of the $\k$-core of a random graph.
\section{The $\k$-core}\label{kcore}
The functions
\begin{equation}\label{fkdef}
f_i(\l) = \sum_{j=i}^{\infty} \frac{\l^j}{j!}, \quad i = 0, 1, 2, \dots,
\end{equation}
figure prominently in our calculations. For $\l > 0$ define 
$$g_i(\l) = \frac{\l f_{2-i}(\l)}{f_{3-i}(\l)}, \quad g_i(0) = 3-i, \quad i = 0,1,2.$$ 
Properties of these functions are derived in Appendix B.

The {\em $\k$-core} $C_\k(G)$ of a graph $G$ is the largest set of vertices that induces a graph $H_\k$ such that the minimum degree $\d(H_\k)\geq \k$. Pittel, Spencer and Wormald \cite{PSW} proved that there exist constants, $c_\k,\k\geq 3$ such that if $p=c/n$ and $c<c_\k$ then w.h.p. $G_{n,p}$ has no $\k$-core and that if $c>c_\k$ then w.h.p. $G_{n,p}$ has a $\k$-core of linear size. We list some facts about these cores that we will need in what follows.  

Given $\l$ let $\mbox{Po}(\l)$ be the Poisson random variable with mean $\l$ and let 
$$\p_r(\l)=\Prob{\mbox{Po}(\l)\geq r}=e^{-\l}f_r(\l).$$ 
Then
\beq{w1}{
c_\k=\inf\brac{\frac{\l}{\p_{\k-1}(\l)}:\l>0}.
}
When $c>c_\k$ define $\l_\k(c)$ by
\beq{w0}{
\l_\k(c)\text{ is the larger of the two roots to the equation }c=\frac{\l}{\p_{\k-1}(\l)}=\frac{\l e^\l}{f_{\k-1}(\l)}.
}
Then \wvhp\footnote{For the purposes of this paper, a sequence of events $\cE_n$ will be said to occur {\em with very high probability }\wvhp\ if $\Prob{\cE_n}=1-o(n^{-1})$. Similarly, $\cE_n$ will be said to occur {\em with high probability }\whp\ if $\Prob{\cE_n}=1-o(1)$.} with $\l=\l_\k(c)$ we have that
\beq{w2}{
C_\k(G_{n,p})\text{ has $\approx \p_\k(\l)n=\frac{f_\k(\l)}{e^\l}n$ vertices and $\approx\frac{\l^2}{2c}n =\frac{\l f_{\k-1}(\l)}{2e^\l}n$ edges}.
}
Furthermore, when $\k$ is large,
\beq{w3}{
c_\k=\k+(\k\log\k)^{1/2}+O(\log\k).
}
{\L}uczak \cite{Lu} proved that $C_{\k}$ is $\k$-connected \wvhp\ when $\k\geq 3$.

Next let $c_\k'$ be the threshold for the $(\k+1)$-core having average degree $2\k$. Here, see \eqref{w0} and \eqref{w2}, 
\beq{ckdash}{
c_\k' = \frac{\l e^{\l}}{f_\k(\l)}\text{ where }\frac{\l f_k(\l)}{f_{k+1}(\l)}=2\k.}
We have $c_2 \approx 3.35$ and $c_2' \approx 3.59$.

\section{Proof of Theorem \ref{th1a}: Large $k$.}
We will prove Theorem \ref{th1a} in this section. It is relatively straightforward. Theorem \ref{th1b} is more involved and occupies Section \ref{(b)}.

In this section we assume that $k=O(1)$ and large. Let $Z_k$ denote the sum of the $k(n-1)$ shortest edge lengths in $K_n$. We have that for $n\gg k$,
\begin{equation}\label{lower}
\mbox{mst}_k(K_n)\geq
\E{Z_k}=\sum_{\ell=1}^{k(n-1)}\frac{\ell}{\binom{n}{2}+1}=\frac{k(n-1)(k(n-1)+1)}{n(n-1)+2}
\in [ k^2(1-n^{-1}),k^2].
\end{equation}
This gives us the lower bound in Theorem \ref{th1a}. 

For the upper bound let $k_0=k+k^{2/3}$ and consider the random graph $H$ generated
by the $k_0(n-1)$ cheapest edges of $K_n$. The expected total edge weight $\overline{E}_H$ of $H$ is at most $k_0^2$, see \eqref{lower}. 

$H$ is distributed as $G_{n,k_0n}$. This is sufficiently close in distribution to $G_{n,p},p=2k_0/n$ that we can apply the results of Section \ref{kcore} without further comment. It follows from \eqref{w3} that $c_{2k}<2k_0$. Putting $\l_0=\l_{2k}(2k_0)$ we see from \eqref{w2} that w.v.h.p. $H$ has a $2k$-core $C_{2k}$ with $\sim n\Prob{\mbox{Po}(\l_0)\geq 2k}$ vertices. It follows from \eqref{w0} that $\l_0=2k_0\p_{2k-1}(2k_0)\leq 2k_0$ and since $\p_{2k-1}(\l)$ increases with $\l$ and $\p_{2k-1}(2k+k^{2/3})=\Prob{\mbox{Po}(2k+k^{2/3})\geq 2k-1}\geq 1-e^{-c_1k^{1/3}}$ for some constant $c_1>0$ we see that $\frac{2k+k^{2/3}}{\p_{2k-1}(2k+k^{2/3})}\leq 2k_0$ and so $\l_0\geq 2k+k^{2/3}$.

A theorem of Nash-Williams \cite{NW} states that a $2k$-edge connected graph contains $k$ edge-disjoint spanning trees. Applying the result of {\L}uczak \cite{Lu} we
see that \wvhp\ $C_{2k}$ contains $k$ edge disjoint spanning trees
$T_1,T_2,\ldots,T_k$. It remains to argue that we can cheaply augment
these trees to spanning trees of $K_n$. Since $|C_{2k}|\sim
n\Prob{\mbox{Po}(\l)\geq 2k}$ \wvhp, we see that \wvhp\ $D_{2k}=[n]\setminus
{C}_{2k}$ satisfies $|D_{2k}|\leq 2n e^{-c_1k^{1/3}}$.

For each $v\in D_{2k}$ we let $S_v$ be the $k$ shortest edges from $v$ to
$C_{2k}$. We can then add $v$ as a leaf to each of the trees
$T_1,T_2,\ldots,T_k$ by using one of these edges. What is the total
weight of the edges $Y_v,\,v\in D_{2k}$? We can bound this
probabilistically by using the following lemma from Frieze and
Grimmett \cite{FG}:
\begin{lemma}\label{lemFG}
Suppose that $k_1+k_2+\cdots+k_M\leq a$, and $Y_1,Y_2,\ldots,Y_M$
are independent random variables with $Y_i$ distributed as the $k_i$th
minimum of $N$ independent uniform [0,1] random variables. If $\m>1$ then
$$\Prob{Y_1+\cdots+Y_M\geq \frac{\m a}{N+1}}\leq
e^{a(1+\ln \m-\m)}.$$
\end{lemma}
Let $\e=2e^{-c_1k^{1/3}}$ and $\m=10\ln 1/\e$ and let $M=k\e
n,\,N=(1-\e)n,\,a=k^2\e n$. Let $\cB_0$ be the event that
there exists a set $S$ of size $\e n$ such that the sum over $v\in S$ of the lengths of the $k$
shortest edges from $v$ to $[n]\setminus S$ exceeds $\m
a/(N+1)$. Next let $\cB$ be the event that the sum over $v\in S$ of the length of the $k$th shortest edge from $v$ to $[n]\setminus S$ exceeds $\m a/(k(N+1))$. We have $\cB_0\subseteq \cB$ and applying Lemma \ref{lemFG} we see
that
$$\Prob{\cB}\leq \binom{n}{\e n}\exp\set{k\e n(1+\ln \m-\m)}\leq
\brac{\frac{e}{\e}\cdot e^{-\mu k/2}}^{\e n}=o(n^{-1}).$$
It follows that
$$\mbox{mst}_k(K_n)\leq o(1)+k_0^2+\frac{\m a}{N+1}\leq k^2+3k^{5/3}.$$
The $o(1)$ term is a bound $kn\times o(n^{-1})$, to account for the
cases that occur with probability $o(n^{-1})$.

Combining this with \eqref{lower} we see that
$$k^2\leq \m_k\leq k^2+3k^{5/3}$$
which proves Theorem \ref{th1a}.
\section{Proof of Theorem \ref{th1b}: $k=2$.}\label{(b)}
For this case we use the fact that for any graph $G=(V,E)$, the collection of subsets
$I\subseteq E$ that can be partitioned into two edge disjoint forests
form the independent sets in a matroid. This being the matroid which
is the union of two copies of the cycle matroid of $G$. See for
example Oxley \cite{Ox} or Welsh \cite{We}. Let $r_2$ denote the rank
function of this matroid, when $G=K_n$. If $G$ is a sub-graph of $K_n$
then $r_2(G)$ is the rank of its edge-set.

We will follow the proof method in \cite{ab92}, \cite{BFM} and \cite{Ja}.
Let $F$ denote the random set of edges in the minimum weight pair of
edge disjoint spanning trees. For any $0 \leq p \leq 1$ let $G_p$ denote the graph induced by
the edges $e$ of $K_n$ which satisfy $X_e\leq p$. Note that $G_p$ is
distributed as $G_{n,p}$.

For any $0 \leq p \leq 1$,
$\sum_{e \in F} 1_{(X_e >p)}$ is the number of edges of $F$ which are not in
$G_p$, which equals $2n-2 - r_2(G_p)$.
So,
$$\mbox{mst}_2(K_n,{\bf X}) = \sum_{e \in F} X_e = \sum_{e \in F} \int_{p=0}^1 1_{(X_e >p)} dp
= \int_{p=0}^1 \sum_{e \in F} 1_{(X_e >p)} dp.$$
Hence, on taking expectations we obtain
\begin{equation}\label{integral}
\mbox{mst}_2(K_n)= \int_{p=0}^1 (2n-2-\E{r_2(G_p)})dp.
\end{equation}
It remains to estimate $\E{r_2(G_p)}$. The main contribution to the integral in \eqref{integral} comes from $p=c/n$ where $c$ is constant. Estimating $\E{r_2(G_p)}$ is easy enough for sufficiently small $c$, but it becomes more difficult for $c>c_2'$, see \eqref{ckdash}. 
When $p=\frac{c}{n}$ for $c>c_k$ we will need to be able to
estimate $\E{r_{k}(C_{k+1}(G_{n,p}))}$. We give  partial results for $k \geq 3$ and complete results for $k = 2$.
We begin with a simple observation.
\begin{lemma}\label{lem1}
Let $k\geq 2$. Let $C_{k+1}=C_{k+1}(G)$ denote the graph induced by the $(k+1)$-core of graph $G$ (it may be an empty sub-graph). Let $E_k(G)$ denote the set of edges that are
{\bf not} contained in $C_{k+1}$. Then 
\begin{equation}\label{E2}
r_k(G)=|E_k(G)|+r_k(C_{k+1}).
\end{equation}
\end{lemma}
\begin{proof}
By induction on $|V(G)|$. Trivial if $|V(G)|=1$ and so assume that
$|V(G)|>1$. If $\d(G)\geq k+1$ then $G=C_{k+1}$ and there is nothing to
prove.
Otherwise, $G$ contains a vertex $v$ of degree
$d_G(v)\leq k$. Now $G-v$ has the same $(k+1)$-core as $G$. If $F_1,...,F_k$ are edge disjoint forests such that
$r_k(G)=|F_1|+ ... + |F_k|$ then by removing $v$ we see, inductively, that $|E_k(G-v)|+r_k(C_{k+1})=r_k(G-v)\geq
|F_1|+...+|F_k|-d_G(v)=r_k(G)-d_G(v)$. On the other hand
$G-v$ contains $k$ forests $F_1',...,F_k'$ such that $r_k(G-v)=|F_1'|+...+|F_k'|=|E_k(G-v)|+r_k(C_{k+1})$. We can then add $v$ as a vertex of degree one to $d_G(v)$ of the forests $F_1',...,F_k'$, implying that $r_k(G)\geq
d_G(v)+|E_k(G-v)|+r_k(C_{k+1})$. Thus, $r_k(G)=d_G(v)+|E_k(G-v)|+r_k(C_{k+1})=|E_k(G)|+r_k(C_{k+1})$.
\end{proof}

\begin{lemma}\label{subcritical}
Let $k\geq 2$. If $c_k < c < c_k'$, then w.h.p. 
\beq{w6}{
|E(G_{n,c/n})|-o(n)\leq r_{k}(G_{n,c/n}) \leq |E(G_{n,c/n})|.}
\end{lemma}
\begin{proof}
We will show that when $c<c_k'$ we can find $k$ disjoint forests $F_1,F_2,\ldots,F_k$ contained in $C_{k+1}$ such that 
\beq{w5}{
|E(C_{k+1})|-\sum_{i=1}^k|E(F_i)|=o(n).}
This implies that $r_k(C_{k+1})\geq |E(C_{k+1})|-o(n)$ and because $r_k(C_{k+1})\leq |E(C_{k+1})|$ the lemma follows from this and Lemma \ref{lem1}.

Gao, P\'erez-Gim\'enez and Sato \cite {gps} show that  when $c<c_k'$, no subgraph of $G_{n, p}$ has average degree more than $2k$, w.h.p. Fix $\e > 0$. Cain, Sanders and Wormald \cite{csw} proved that if the average degree of the $(k+1)$-core is at most $2k - \e$, then w.h.p. the edges of $G_{n, p}$ can be oriented so that no vertex has indegree more than $k$. It is clear from \eqref{w2} that the edge density of the $(k+1)$-core increases smoothly w.h.p. and so we can apply the result of \cite{csw} for some value of $\e$.

It then follows that the edges of $G_{n,p}$ can be partitioned into $k$ sets $\F_1,\F_2,\ldots,\F_k$ where each subgraph $H_i=([n],\F_i)$ can be oriented so that each vertex has indegree at most one. We call such a graph a {\em Partial Functional Digraph} or PFD. Each component of a PFD is either a tree or contains exactly one cycle. We obtain $F_1,F_2,\ldots,F_k$ by removing one edge from each such cycle. We must show that w.h.p. we remove $o(n)$ vertices in total. Observe that if $Z$ denotes the number of edges of $G_{n,p}$ that are on cycles of length at most $\om_0=\frac13\log_cn$ then
$$\E{Z}\leq \sum_{\ell=3}^{\om_0}\ell!\binom{n}{\ell}\ell p^\ell\leq \om_0c^{\om_0}\leq n^{1/2}.$$
The Markov inequality implies that $Z\leq n^{2/3}$ w.h.p. The number of edges removed from the larger cycles to create $F_1,F_2,\ldots,F_k$ can be bounded by $kn/\om_0=o(n)$ and this proves \eqref{w5} and the lemma.
\end{proof}

\begin{lemma}\label{supercritical}
If $c > c_2'$, then w.h.p. the $3$-core of $G_{n, c/n}$ contains two edge-disjoint forests of total size $2|V(C_3)| - o(n)$. In particular, $r_2(C_3(G_{n, c/n})) = 2 |V(C_3)| - o(n)$.
\end{lemma}

The proof of Lemma \ref{supercritical} is postponed to Section \ref{sec:super}. We can now prove Theorem \ref{th1b}.

\section{Proof of Theorem \ref{th1b}.}
As noted in (\ref{integral}),
\begin{equation}
\mbox{mst}_2(K_n)= \int_{p=0}^1 (2n-2-\E{r_2(G_p)})dp.
\end{equation}
A crude calculation shows that if $c$ is large then
\begin{equation}\label{crude}
p\geq \frac{c}{n}\text{ implies that }\Prob{r_2(G_p)<2n-2-nAc^6e^{-c}}=o(1),
\end{equation}
for some absolute constant $A>0$.

Indeed, we know that if $p=\frac{c}{n}$ and $c$ is suficently large, then $G_p$ contains a pair of edge disjoint cycles, each of length at least $n(1-c^6e^{-c})$  with probability $1-\e_1$, where $\e_1=O(n^{-\a})$, for some absolute constant $\a>0$, see Frieze \cite{F}. If $p_1=\frac{c_1}{n}$ and $p_2=Kp_1$ then $\Prob{r_2(G_{p_2})<2n-2-nc^6e^{-c}}\leq \e_1^{p_2/p_1}=O(n^{-K\a})$ since $G_{p_2}$ can be generated by adding edges to $p_2/p_1$ independent copies of $G_{p_1}$. This confirms \eqref{crude}.

So, for large $c$,
\begin{equation}\label{cr1}
\mbox{mst}_2(K_n)= \int_{p=0}^{\frac{c}{n}} (2n-2-\E{r_2(G_p)})dp+\e_c,
\end{equation}
where 
$$0\leq \e_c\leq An\int_{p=\frac{c}{n}}^1(np)^6e^{-np}dp=A\int_{x=c}^nx^6e^{-x}dx\leq A\int_{x=c}^\infty x^6e^{-x}dx<c^7e^{-c},$$
after changing variables to $x = pn$. Doing this once more we have,
\begin{equation}
\mbox{mst}_2(K_n) = \int_{x = 0}^{c} \brac{2 - 2n^{-1} - n^{-1}\E{r_2(G_{\frac{x}{n}})}}dx+\e_c.
\end{equation}

By Lemmas \ref{lem1} and \ref{subcritical}, for $x < c_2'$ we have $n^{-1}\E{r_2(G_{\frac{x}{n}})} = n^{-1}\E{|E(G_{\frac{x}{n}})|} - \xi(x,n)= x/2 - \xi(x,n)$ where $\lim_{n\to\infty}\xi(x,n)=0$. Now $n^{-1}\E{r_2(G_{\frac{x}{n}})},n=1,2,\ldots,$ is a sequence of bounded monotone increasing continuous functions of $x$. This sequence converges pointwise to a continuous function $f$, and it actually converges uniformly to $f$. Thus we can bound $\max_{0\leq x\leq c_2'}\xi(x,n)\leq \eta(n)$ where $\lim_{n\to\infty}\eta(n)=0$.  Clearly $f(x)=x/2$ and so
$$\int_{x=0}^{c_2'}n^{-1}\E{r_2(G_{\frac{x}{n}})}dx=\int_{x=0}^{c_2'}\frac{x}{2}dx+o(1). $$
By Lemma \ref{supercritical}, for $x > c_2'$ we have $\E{r_2(C_3(G_{\frac{x}{n}}))} =\E{ 2|V(C_3)|} - o(n)$. So by Lemma \ref{lem1} $\E{r_2(G_{\frac{x}{n}})} = \E{|E(G_{\frac{x}{n}})| - |E(C_3)| + 2|V(C_3)|} - o(n)$, and
\begin{equation}
\m_2 = \int_{x = 0}^{c_2'} \brac{2 - \frac{x}{2}}dx + \int_{x=c_2'}^{c} \brac{2 - \frac{1}{n}\left(\frac{xn}{2} -\E{ |E(C_3(G_{\frac{x}{n}}))} + \E{2|V(C_3(G_{\frac{x}{n}}))|}\right)} dx + \e_c+o(1)
\end{equation}

We have from \eqref{w2} that for $p = x / n$ we have
\begin{eqnarray*}
\frac{1}{n}\E{|V(C_3)|} &=& \frac{f_3(\l)}{e^\l} + o(1) \\
\frac{1}{n}\E{|E(C_3)|} &=& \frac{\l f_2(\l)}{2e^\l} + o(1)
\end{eqnarray*}
where $\l$ is the largest solution to $\l e^{\l} / f_2(\l) = x$. 
Thus,
\begin{equation}
\mu_2 =\lim_{n \rightarrow \infty} \mbox{mst}_2(K_n) = \int_{x = 0}^{c_2'}\brac{ 2 - \frac{x}{2}} dx + \int_{x=c_2'}^c \brac{2 - \frac{x}{2} + \frac{\l f_2(\l)}{2e^\l} - 2\frac{f_3(\l)}{e^\l}}  dx+\e_c.
\end{equation}
To calculate this, note that
\begin{equation}
\frac{dx}{d\l} = \frac{e^\l}{f_2(\l)} + \frac{\l e^\l}{f_2(\l)} - \frac{\l e^\l f_1(\l)}{f_2(\l)^2}
\end{equation}
so
\begin{eqnarray}
&& \int_{x=c_2'}^c \brac{2 - \frac{x}{2} + \frac{\l f_2(\l)}{2e^\l} - 2\frac{f_3(\l)}{e^\l} } dx \nonumber\\
&=& \int_{\l(c_2')}^{\l(c)} \left(2 - \frac{\l e^\l}{2f_2(\l)} + \frac{\l f_2(\l)}{2e^\l} - 2\frac{f_3(\l)}{e^\l} \right) \left(\frac{e^\l}{f_2(\l)} + \frac{\l e^\l}{f_2(\l)} - \frac{\l e^\l f_1(\l)}{f_2(\l)^2}\right) d\l+\e_c\label{intc}
\end{eqnarray}
where $\l(x)$ is the unique solution to $\l e^{\l} / f_2(\l) = x$. 

Note that 
\begin{equation}\label{l2pdef}
\l(c_2') \approx 2.688\text{ and }\l(c)>\frac{c}{2}\text{ for large }c. 
\end{equation}
Now for large $\l$ we can bound 
$$\left(2 - \frac{\l e^\l}{2f_2(\l)} + \frac{\l f_2(\l)}{2e^\l} - 2\frac{f_3(\l)}{e^\l} \right) \left(\frac{e^\l}{f_2(\l)} + \frac{\l e^\l}{f_2(\l)} - \frac{\l e^\l f_1(\l)}{f_2(\l)^2}\right)$$
from above by $\l^3e^{-\l}$. So the range in the integral in \eqref{intc} can be extended to $\infty$ at the cost of adding an amount $\d_c$ where $0\leq \d_c\leq c^4e^{-c}$. Using the fact that we can make $\e_c,\d_c$ arbitrarily close to zero by making $c$ abritrarily large, we obtain the expression for $\m_2$ claimed in Theorem \ref{th1b}.

Attempts to transform the integral in the theorem into an explicit integral with explicit bounds have been unsuccesful. Numerical calculations give
\begin{equation}
\mu_2 \approx 4.1704288 \dots
\end{equation}
The Inverse Symbolic Calculator\footnote{https://isc.carma.newcastle.edu.au/} has yielded no symbolic representation of this number. An apparent connection to the $\z$ function lies in its representation as
\begin{equation}
\z(x) = \frac{1}{\G(x)}\int_{\l=0}^\infty \frac{\l^{x-1}}{e^\l - 1} d\l
\end{equation}
which is somewhat similar to terms of the form
\begin{equation}
\int_{\l=\l_2'}^\infty \frac{\mbox{poly}(\l)}{e^\l - 1 - \l} d\l
\end{equation}
appearing in $\mu_2$, but no real connection has been found.

\section{Proof of Lemma \ref{supercritical}.} \label{sec:super}
\subsection{More on the 3-core.}
Suppose now that $c>c_2'$ and that the 3-core $C_3$ of $G_{n,p}$ has $N=\Omega(n)$ vertices and $M$ edges. It will be distributed as a random graph uniformly chosen from the set of graphs with vertex set $[N]$ and $M$ edges and minimum degree at least three. This is an easy well known observation and follows from the fact that each such graph $H$ can be extended in the same number of ways to a graph $G$ with vertex set $[n]$ and $m$ edges and such that $H$ is the 3-core of $G$. We will for convenience now assume that $V(C_3)=[N]$. 

The degree sequence $d(v),v\in [N]$ can be generated as follows: We independently choose for each $v\in V(C_3)$ a truncated Poisson random variable with parameter $\l$ satisfying $g_0(\l) = 2 M / N$, conditioned on $d(v) \geq 3$. So for $v \in [N]$,
\begin{equation}
\Prob{d(v) = k} = \frac{\l^k}{k! f_3(\l)}, \quad k = 3, 4, 5, \dots, \quad \l = g_0^{-1}\left(\frac{2M}{N}\right)
\end{equation}
Properties of the functions $f_i, g_i$ are derived in Appendix B. In particular, the $g_i$ are strictly increasing by Lemma \ref{gconvex}, so $g_0^{-1}$ is well defined.

These independent variables are further conditioned so that the event 
\beq{cD1}{
\cD=\set{\sum_{v\in [N]}d(v)=2M}}
occurs. Now $\l$ has been chosen so that $\E{d(v)}=2M/N$ and then the local central limit theorem implies that $\Prob{\cD}=\Omega(1/N^{1/2})$, see for example Durrett \cite{Dur}. It follows that
\beq{cD}{
\Prob{\cE\mid\cD}\leq O(n^{1/2})\Prob{\cE},}
for any event $\cE$ that depends on the degree sequence of $C_3$.

In what follows we use the configuration model of Bollob\'as \cite{Bo1} to analyse $C_3$ after we have fixed its degree sequence. Thus, for each vertex $v$ we define a set $W_v$ of \emph{points} such that $|W_v| = d(v)$, and write $W = \bigcup_v W_v$. A random configuration $F$ is generated by selecting a random partition of $W$ into $M$ pairs. A pair $\set{x,y}\in F$ with $x\in W_u,y\in W_v$ yields an edge $\set{u,v}$ of the associated (multi-)graph $\G_F$.

The key properties of $F$ that we need are (i) conditional on $F$ having no loops or multiple edges, it is equally likely to be any simple graph with the given degree sequence and (ii) for the degree sequences of interest, the probability that $\G_F$ is simple will be bounded away from zero. This is because the degree sequence in \eqref{cD} has exponential tails. Thus we only need to show that $\G_F$ has certain properties w.h.p.
\subsection{Setting up the main calculation.}
Suppose now that $p=c/n$ where $c>c_2'$. We will show that w.h.p., for any fixed $\e>0$,
\begin{equation} \label{good}
i(S) = |\{e \in E(C_3) : e\cap S \neq \emptyset\}| \geq (2-\e)|S|\text{ for all }S\subseteq [N]. 
\end{equation}
Proving this is the main computational task of the paper. In principle, it is just an application of the first moment method. We compute the expected number of $S$ that violate \eqref{good} and show that tis expectation tends to zero. On the other hand, a moments glance at the expression $f(\textbf{w})$ below shows that this is unlikely to be easy and it takes more than half of the paper to verify \eqref{good}. 

It follows from \eqref{good} that 
\beq{w10}{
\text{$E(C_3)$ can be oriented so that at least $(1-\e)N$ vertices have indegree at least two.}}
To see this consider the following network flow problem. We have a source $s$ and a sink $t$ plus a vertex for each $v \in [N]$ and a vertex for each edge $e\in E(C_3)$. The directed edges are (i) $(s,v),v\in [N]$ of capacity two; (ii) $(u,e)$, where $u\in e$ of infinite capacity; (iii) $(e,t), e\in E(C_3)$ of capacity one. A $s-t$ flow decomposes into paths $s,u,e,t$ corresponding to orienting the edge $e$ into $u$. A flow thus corresponds to an orientation of $E(C_3)$. The condition \eqref{good} implies that the minimum cut in the network has capacity at least $(2-\e)N$. This implies that there is a flow of value at least $(2-\e)N$ and then the orientation claimed in \eqref{w10} exists.

Thus w.h.p. $C_3$ contains two edge-disjoint PFD's, each containing $(1-\e)N$ edges. Arguing as in the proof of Lemma \ref{subcritical}, we see that we can w.h.p. remove $o(N)$ edges from the cycles of these PFD's and obtain forests. Thus w.h.p. $C_3$ contains two edge-disjoint forests of total size at least $2(1-\e)N-o(N)$. This implies that $\E{r_2(C_3(G_{n, c/n}))} \geq 2(1-\e)N - o(N)$ and since $N = \Omega(n)$, we can have $\E{r_2(C_3(G_{n, c/n}))}=2(1-\e)N-o(n)$. Because $\e$ is arbitrary, this
implies $r_2(C_3(G_{n, c/n})) = 2N - o(n)$ whenever $c > c_2'$. 
\subsection{Proof of \eqref{good}: Small $S$.}
It will be fairly easy to show that \eqref{w10} holds w.h.p. for all $|S|\leq s_\e$ where
$$
s_\e=\bfrac{1+\e}{e^{2+\e}c}^{1/\e}n.
$$
We claim that w.h.p.
\beq{w9}{
|S|\leq s_\e\text{ implies }e(S)<(1+\e)|S|\text{ in }G_{n,p}.
}
Here $e(S)=|\set{e\in E(G_{n,p}):e\subseteq S}|.$

Indeed,
\begin{multline*}
\Prob{\exists S\text{ violating }\eqref{w9}}\leq \sum_{s=4}^{s_\e}\binom{n}{s}\binom{\binom{s}{2}}{(1+\e)s}p^{(1+\e)s}\leq\\ \sum_{s=4}^{s_\e}\bfrac{ne}{s}^s\bfrac{sec}{2(1+\e)n}^{(1+\e)s}= \sum_{s=4}^{s_\e}\brac{\bfrac{s}{n}^\e\frac{e^{2+\e}c}{2(1+\e)}}^s=o(1).
\end{multline*}
For sets $A,B$ of vertices and $v\in A$ we will let $d_B(v)$ denote the number of neighbors of $v$ in $B$. We then let $d_B(A) = \sum_{v \in A} d_B(v)$. We will drop the subscript $B$ when $B=[N]$.

Suppose then that \eqref{w9} holds and that $|S|\leq s_\e$ and $i(S)\leq (2-\e)|S|$. Then if $\bar{S}=[N]\setminus S$, we have
$$e(S)+d_{\bar{S}}(S)\leq (2-\e)|S|\text{ and }d(S)=2e(S)+d_{\bar{S}}(S)\geq 3|S|$$
which implies that $e(S)\geq (1+\e)|S|$, contradiction.
\subsection{Proof of \eqref{good}: Large $S$.}
Suppose now that $C_3$ contains an $S$ such that $i(S) < (2-\e)|S|$. Let such sets be {\em bad}. Let $S$ be a minimal bad set, and write $T = [N] \setminus S$. For any $v \in S$, we have $i(S \setminus v) \geq (2-\e)|S\setminus v|$ while $i(S) < (2-\e) |S|$. This implies $d_T(v) = i(S) - i(S \setminus v) < 2$. 

We will start with a minimal bad set and then carefully add more vertices. Consider a set $S$ such that $i(S) < 2 |S|$ and $d_T(v) \leq 2$ for all $v \in S$. If there is a $w \in T$ such that $d_T(w) \leq 2$, let $S' = S \cup \{w\}$. We have $i(S') \leq i(S) + 2 < 2|S'|$. This means we may add vertices to $S$ in this fashion to aquire a partition $[N] = S \cup T$ where $d_T(v) \leq 2$ for all $v \in S$ and $d_T(v) \geq 3$ for all $v \in T$. We further partition $S = S_0 \cup S_1 \cup S_2$ so that $d_T(v) = i$ if and only if $v \in S_i$. Denote the size of any set by its lower case equivalent, i.e. $|S_i| = s_i$ and $|T| = t$.

We now start to use the configuration model. Partition each point set into $W_v = W_v^S \cup W_v^T$, where a point is in $W_v^S$ if and only if it is matched to a point in $\cup_{u \in S} W_u$. The sizes of $W_v^S, W_v^T$ uniquely determine $\textbf{w} = (s_0, s_1, s_2, D_0, D_1, D_2, D_3, t, M)$. Here $D_i = d_S(S_i), i = 0,1,2$ and $D_3 = d_T(T)$.

\subsubsection{Estimating the probability of $\textbf{w}$.}
By construction, $D_i \geq (3 - i)s_i$ for $i = 0, 1, 2$ and $D_3 \geq 3t$. Define degree sequences $(d_i^1, \dots, d_i^{s_i})$ for $S_i, i = 0, 1, 2$ and $(d_3^1, \dots, d_3^t)$ for $T$. Furthermore, let $\widehat{d}_1^j = d_1^j - 1$, $\widehat{d}_2^j = d_2^j - 2$ and $\widehat{d}_3^j \geq 0$ be the $S$-degrees of vertices in $S_1, S_2, T$, respectively.

{\bf Dealing with $S_0$:}\\
Ignoring for the moment, that we must condition on the event $\cD$ (see \eqref{cD1}), the probability that $S_0$ has degree sequence $(d_0^1, \dots, d_0^{s_0})$, $d_0^i \geq 3$ for all $i$, is given by
\begin{equation}
\prod_{i=1}^{s_0} \frac{\l^{d_0^i}}{d_0^i ! f_3(\l)}
\end{equation}
where $\l$ is the solution to 
$$g_0(\l) = \frac{2M}{N}.$$ 
Hence, letting $[x^D] f(x)$ denote the coefficient of $x^D$ in the power series $f(x)$, the probability $\p_0(S_0,D_0)$ that $d(S_0)=D_0$ is bounded by
\begin{eqnarray}
\p_0(S_0,D_0)\leq \sum_{\stackrel{d_0^1 + \dots + d_0^{s_0} = D_0}{d_0^i \geq 3}} \prod_{i=1}^{s_0} \frac{\l^{d_0^i}}{d_0^i ! f_3(\l)} &=&  \frac{\l^{D_0}}{f_3(\l)^{s_0}} \sum_{\stackrel{d_0^1 + \dots + d_0^{s_0} = D_0}{d_0^i \geq 3}} \prod_{i=1}^{s_0} \frac{1}{d_0^i !} \nonumber\\
&=& \frac{\l^{D_0}}{f_3(\l)^{s_0}} [x^{D_0}] \left(\sum_{d_0 \geq 3} \frac{x^{d_0}}{d_0!}\right)^{s_0}\nonumber \\
&=& \frac{\l^{D_0}}{f_3(\l)^{s_0}} [x^{D_0}] f_3(x)^{s_0}\nonumber \\
&\leq& \frac{\l^{D_0}}{f_3(\l)^{s_0}} \frac{f_3(\l_0)^{s_0}}{\l_0^{D_0}}\label{expr1}
\end{eqnarray}
for all $\l_0$. Here we use the fact that for any function $f$ and any $y > 0$, $[x^{D_0}] f(x) \leq f(y) / y^{D_0}$. To minimise \eqref{expr1} we choose $\l_0$ to be the unique solution to  
\beq{l0}{
g_0(\l_0) = \frac{D_0}{s_0}. 
}
If $D_0 = 3s_0$ then $\l_0 = 0$ by Lemma \ref{gibounds}, Appendix B. In this case, since $f_3(\l_0)=\frac{\l_0^3(1+O(\l_0))}{6}$, we have
\begin{equation}\label{lambdazero}
\p_0(S_0,D_0)\leq \bfrac{\l^3}{6f_3(\l)}^{s_0}, \quad\text{when }D_0=3s_0. 
\end{equation}
{\bf Dealing with $S_1$:}\\
For each $v \in S_1$, we have $W_v = W_v^S \cup W_v^T$ where $|W_v^T| = 1$. Hence, the probability $\p_1(S_1,D_1)$ that $d(S_1)=D_1 + s_1$ is bounded by
\begin{eqnarray}
\p_1(S_1,D_1)\leq \sum_{\stackrel{\widehat{d}_1^1 + \dots + \widehat{d}_1^{s_1} = D_1}{\widehat{d}_1^i \geq 2}} \prod_{i=1}^{s_1} \binom{\widehat{d}_1^i + 1}{1} \frac{\l^{\widehat{d}_1^i + 1}}{(\widehat{d}_1^i + 1) ! f_3(\l)} \nonumber 
&=& \frac{\l^{D_1 + s_1}}{f_3(\l)^{s_1}} \sum_{\stackrel{\widehat{d}_1^1 + \dots + \widehat{d}_1^{s_1} = D_1}{\widehat{d}_1^i \geq 2}} \prod_{i=1}^{s_1}  \frac{1}{\widehat{d}_1^i ! }\\
&=& \frac{\l^{D_1 + s_1}}{f_3(\l)^{s_1}} [x^{D_1}]f_2(x)^{s_1} \nonumber\\
&\leq& \frac{\l^{D_1 + s_1}}{f_3(\l)^{s_1}} \frac{f_2(\l_1)^{s_1}}{\l_1^{D_1}}.\label{expr2}
\end{eqnarray}
We choose $\l_1$ to satisfy the equation 
\beq{l1}{
g_1(\l_1) = \frac{D_1}{s_1}.
} 
Similarly to what happens in (\ref{lambdazero}) we have $\l_1=0$ when $D_1=2s_1$ and $f_2(\l_1)=\frac{\l_1^2(1+O(\l_1))}{2}$, so
\beq{S1case}{
\p_1(S_1,D_1)\leq \bfrac{\l^3}{2f_3(\l)}^{s_1}, \quad\text{when }D_1=2s_1.}

{\bf Dealing with $S_2$:}\\
For $v \in S_2$, we choose $2$ points from $W_v$ to be in $W_v^T$ , so the probability $\p_2(S_2,D_2)$ that $d(S_2)=D_2 + 2s_2$ is bounded by
\begin{equation}\label{expr3}
\p_2(S_2,D_2)\leq \sum_{\stackrel{\widehat{d}_2^1 + \dots + \widehat{d}_2^{s_2} = D_2}{\widehat{d}_2^i \geq 1}} \prod_{i=1}^{s_2} \binom{\widehat{d}_2^i + 2}{2} \frac{\l^{\widehat{d}_2^i + 2}}{(\widehat{d}_2^i + 2)! f_3(\l)} \leq \frac{\l^{D_2 + 2s_2}}{f_3(\l)^{s_2}} \frac{f_1(\l_2)^{s_2}}{\l_2^{D_2}} 2^{-s_2}
\end{equation}
where we choose $\l_2$ to satisfy the equation
\beq{l2}{
g_2(\l_2) = \frac{D_2}{s_2}.}
Similarly to what happens in (\ref{lambdazero}) we have $\l_2=0$ when $D_2=s_2$ and $f_1(\l_2)=\l_2(1+O(\l_2))$, so
\beq{S2case}{
\p_2(S_2,D_2)\leq \bfrac{\l^3}{2f_3(\l)}^{s_2}, \quad\text{when }D_2=s_2.}
{\bf Dealing with $T$:}\\
Finally, the degree of vertex $i$ in $T$ can be written as $d_3^i = \widehat{d}_3^i + \overline{d}_3^i$ where $\widehat{d}_3^i \geq 0$ is the $S$-degree and $\overline{d}_3^i \geq 3$ is the $T$-degree. Here, with $t=|T|$, we have
$$\sum_{i=1}^t \widehat{d}_3^i = d_S(T) = s_1 + 2s_2$$ 
by the definition of $S_0, S_1, S_2$. 
So the probability $\p_3(T,D_3)$ that $d_T(T)=D_3$, given $s_1,s_2$ can be bounded by
\begin{align}
\p_3(T,D_3)&\leq\sum_{\substack{\widehat{d}_3^1+\cdots+\widehat{d}_3^t=s_1+2s_2\\\widehat{d}_3^i\geq 0}} \ \ 
\sum_{\substack{\overline{d}_3^1+\cdots+\overline{d}_3^t=D_3\\\overline{d}_3^i\geq 3}} 
\prod_{i=1}^t \binom{\widehat{d}_3^i + \overline{d}_3^i}{\widehat{d}_3^i} \frac{\l^{\widehat{d}_3^i + \overline{d}_3^i}}{(\widehat{d}_3^i + \overline{d}_3^i)! f_3(\l)}\nonumber\\
&= \frac{\l^{D_3 + s_1 + 2s_2}}{f_3(\l)^t} \sum_{\substack{\widehat{d}_3^1+\cdots+\widehat{d}_3^t=s_1+2s_2\\\widehat{d}_3^i\geq 0}} \ \ 
\sum_{\substack{\overline{d}_3^1+\cdots+\overline{d}_3^t=D_3\\\overline{d}_3^i\geq 3}} \prod_{i=1}^t \frac{1}{\widehat{d}_3^i ! \overline{d}_3^i !} \nonumber\\
&= \frac{\l^{D_3 + s_1 + 2s_2}}{f_3(\l)^t} \left([x^{D_3}] f_3(x)^t \right) \left([x^{s_1 + 2s_2}] e^x \right) \nonumber\\
&\leq\frac{\l^{D_3 + s_1 + 2s_2}}{f_3(\l)^t} \frac{f_3(\l_3)^{t}}{\l_3^{D_3}} \frac{t^{s_1 + 2s_2}}{(s_1 + 2s_2)!}, \label{expr4}
\end{align}
where we choose $\l_3$ to satisfy the equation 
\beq{l3}{
g_0(\l_3) = \frac{D_3}{t}.
}
Similarly to what happens in (\ref{lambdazero}) we have $\l_3=0$ when $D_3=3t$ and  $f_3(\l_3)=\frac{\l_3^3(1+O(\l_1))}{6}$, so
$$\p_3(T,D_3)\leq \frac{\l^{D_3 + s_1 + 2s_2}}{(6f_3(\l))^t}\frac{t^{s_1 + 2s_2}}{(s_1 + 2s_2)!}, \quad\text{when }D_3=3t.$$
\subsubsection{Putting the bounds together.}
For a fixed $\textbf{w} = (s_0, s_1, s_2, D_0, D_1, D_2, D_3, t, M)$, there are $\binom{t + s}{s_0, s_1, s_2, t}$ choices for $S_0,S_1,S_2,T$. Having chosen these sets we partition the $W_v,v\in S\cup T$ into $W_v^S\cup W_v^T$. Note that our expressions \eqref{expr1}, \eqref{expr2}, \eqref{expr3}, \eqref{expr4} account for these choices.  Given the partitions of the $W_v$'s, there are $(D_0 + D_1 + D_2)!! D_3!! (s_1 + 2s_2)!$ configurations, where $(2s)!!=(2s-1)\times(2s-3)\times\cdots\times3\times1$ is the number of ways of partitioning a set of size $2s$ into $s$ pairs. Here $(D_0 + D_1 + D_2)!!$ is the number of ways of pairing up $\bigcup_{v\in S}W_v^S$, $D_3!!$ is the number of ways of pairing up $\bigcup_{v\in T}W_v^T$ and $(s_1 + 2s_2)!$ is the number of ways of pairing points associated with $S$ to points associated with $T$. Each configuration has probability $1 / (2M)!!$. So, the total probability of all configurations whose vertex partition and degrees are described by $\textbf{w}$ can be bounded by
\begin{align*}
&\binom{t+s}{s_0, s_1, s_2, t} \frac{\l^{D_0}}{f_3(\l)^{s_0}} \frac{f_3(\l_0)^{s_0}}{\l_0^{D_0}} \frac{\l^{D_1 + s_1}}{f_3(\l)^{s_1}} \frac{f_2(\l_1)^{s_1}}{\l_1^{D_1}} \frac{\l^{D_2 + 2s_2}}{f_3(\l)^{s_2}} \frac{f_1(\l_2)^{s_2}}{\l_2^{D_2}} 2^{-s_2} \\
&\times  \frac{\l^{D_3 + s_1 + 2s_2}}{f_3(\l)^t} \frac{f_3(\l_3)^{t}}{\l_3^{D_3}} \frac{t^{s_1 + 2s_2}}{(s_1 + 2s_2)!} \frac{(D_0 + D_1 + D_2)!! D_3!! (s_1 + 2s_2)!}{(2M)!!} \\
=& \binom{t+s}{s_0, s_1, s_2, t} \frac{\l^{2M}}{f_3(\l)^{N}} \frac{f_3(\l_0)^{s_0}}{\l_0^{D_0}}  \frac{f_2(\l_1)^{s_1}}{\l_1^{D_1}}  \frac{f_1(\l_2)^{s_2}}{\l_2^{D_2}} 2^{-s_2}  \frac{f_3(\l_3)^{t}}{\l_3^{D_3}} \frac{t^{s_1 + 2s_2}}{(s_1 + 2s_2)!} \\
& \times \frac{(D_0 + D_1 + D_2)!! D_3!! (s_1 + 2s_2)!}{(2M)!!} \\
\end{align*}

Write $D_i = \D_i s$, $|S_i| = \s_i s$, $t = \tau s$, $M = \m s$ and $N = \n s$. We have $k!! \sim \sqrt{2}(k / e)^{k/2}$ as $k \rightarrow \infty$ by Stirling's formula, so the expression above, modulo an $e^{o(s)}$ factor, can be written as 
\begin{multline}\label{f!!}
f(\textbf{w})^s = \left(\frac{(\t+1)^{\t+1}}{\s_0^{\s_0}\s_1^{\s_1}(1-\s_0-\s_1)^{1 - \s_0 - \s_1} \t^{\t}}\frac{\l^{2\m}}{f_3(\l)^\n}\frac{f_3(\l_0)^{\s_0}}{\l_0^{\D_0}} \frac{f_2(\l_1)^{\s_1}}{\l_1^{\D_1}} \frac{f_1(\l_2)^{\s_2}}{\l_2^{\D_2}} \frac{f_3(\l_3)^{\t}}{\l_3^{\D_3}} \frac{(\t e)^{\s_1 + 2\s_2}}{2^{\s_2}}  \right. \\
 \left. \frac{(\D_0 + \D_1 + \D_2)^{(\D_0 + \D_1 + \D_2) / 2} \D_3^{\D_3 / 2}}{(2\m)^\m}        \right)^s
\end{multline}

We note that 
\begin{align}
\s_2 &= 1 - \s_0 - \s_1,\label{s2}\\
\D_3 &= 2\m - \D_0 - \D_1 - \D_2 - 2\s_1 - 4\s_2\nonumber\\
&= 2\m - 4 - \D_0 - \D_1 - \D_2 + 4\s_0 + 2\s_1\label{D3=}\\
\n &= 1 + \t.\label{nu=}
\end{align}
Hence $\s_2, \D_3, \n$ may be eliminated, and we can consider $\textbf{w}$ to be $(\s_0, \s_1, \D_0, \D_1, \D_2,\t, \m)$. When convenient, $\D_3$ may be used to denote $2\m - 4 - \D_0 - \D_1 - \D_2 + 4\s_0 + 2\s_1$. Define the constraint set $F$ to be all $\textbf{w}$ satisfying 
\begin{align*}
&\D_0 \geq 3\s_0, \D_1 \geq 2\s_1, \D_2 \geq 1 - \s_0 - \s_1, \D_3 \geq 3\t.\\
&\frac{\D_0 + \D_1 + \D_2}{2} + \s_1 + 2(1 - \s_0 - \s_1) < 2-\e & \mbox{ since $i(S) < (2-\e)|S|$, \qquad see (\ref{good})}.\\ 
& \s_0, \s_1 \geq 0, \s_0 + \s_1 \leq 1. \\
&0 \leq \t \leq (1 - \e) / \e  \mbox{ since $|S| \geq \e N$}.\\ 
&\m \geq (2 + \e)(1 + \t)  \mbox{ since $M \geq (2 + \e)N$.}\\
&\s_0<1, \qquad\text{ otherwise $C_3$ is not connected}.
\end{align*}
Here $\e$ is a sufficiently small positive constant such that we can (i) exclude the case of small $S$, (ii) satisfy condition \eqref{good} and (iii) have $M \geq (2 + \e)N$ since $c>c_2'$.

For a given $s$, there are $O(\mbox{poly}(s))$ choices of $\textbf{w} \in F$, and the probability that the randomly chosen configuration corresponds to a $\textbf{w} \in F$ can be bounded by
\beq{fw}{
\sum_{s \geq \e N} \sum_{\textbf{w}} O(\mbox{poly}(s)) f(\textbf{w})^s \leq \sum_s (e^{o(1)}\max_F f(\textbf{w}))^s \leq N (e^{o(1)}\max_F f(\textbf{w}))^{\e N}.}
As $N \rightarrow \infty$, it remains to show that $f(\textbf{w}) \leq 1 - \d$ for all $\textbf{w} \in F$, for some $\d = \d(\e) > 0$. At this point we remind the reader that we have so far ignored conditioning on the event $\cD$ defined in \eqref{cD1}. Inequality \eqref{cD} implies that it is sufficient to inflate the RHS of \eqref{fw} by $O(n^{1/2})$ to obtain our result.

So, let
\begin{eqnarray*}
&& f(\D_0, \D_1, \D_2, \s_0, \s_1, \t, \m) =  \\ 
&& \frac{(\t+1)^{\t+1}}{\s_0^{\s_0} \s_1^{\s_1} (1 - \s_0 - \s_1)^{1 - \s_0 - \s_1} \t^\t} \frac{\l^{2\m}}{f_3(\l)^{\t+1}} \frac{f_3(\l_0)^{\s_0}}{\l_0^{\D_0}} \frac{f_2(\l_1)^{\s_1}}{\l_1^{\D_1}} \frac{f_1(\l_2)^{1-\s_0-\s_1}}{\l_2^{\D_2}} \frac{f_3(\l_3)^\t}{\l_3^{\D_3}} \\
&\times& \frac{(e\t)^{2-2\s_0-\s_1}}{2^{1-\s_0-\s_1}} \frac{(\D_0 + \D_1 + \D_2)^{(\D_0 + \D_1 + \D_2) / 2} \D_3^{\D_3/2}}{(2\m)^\m}
\end{eqnarray*}
We complete the proof of Theorem \ref{th1b} by showing that 
\beq{f()<1}{
f(\textbf{w}) \leq \exp\set{- \frac{\e^2}{3}} \text{for all $\textbf{w} \in F$.}}
The proof of \eqref{f()<1} is a very long careful calculation and we have placed it in Section \ref{app1} of the appendix.
\section{Final Remarks}
There are a number of loose ends to be taken care of. Is Conjecture \ref{conj1} true? Is there a simpler expression for $\m_2$ of Theorem \ref{th1b}? Is it possible to get an exact expression for $\m_3$? On another tack, what are the expected running times of algorithms for computing these edge disjoint trees? They are polynomial time solvable problems, in the worst-case, but maybe their average complexity is significantly better than worst-case.

\appendix
\section{Proof of \eqref{f()<1}}\label{app1}
We remind the reader that the aim of this section is to show that $f(\textbf{w})<1$ for all\\ $\textbf{w}=(\s_0, \s_1, \D_0, \D_1, \D_2,\t, \m)\in F$, where 
\begin{align*}
f(\textbf{w})& =   
 \frac{(\t+1)^{\t+1}}{\s_0^{\s_0} \s_1^{\s_1} (1 - \s_0 - \s_1)^{1 - \s_0 - \s_1} \t^\t} \frac{\l^{2\m}}{f_3(\l)^{\t+1}} \frac{f_3(\l_0)^{\s_0}}{\l_0^{\D_0}} \frac{f_2(\l_1)^{\s_1}}{\l_1^{\D_1}} \frac{f_1(\l_2)^{1-\s_0-\s_1}}{\l_2^{\D_2}} \frac{f_3(\l_3)^\t}{\l_3^{\D_3}} \\
&\hspace{1in}\times \frac{(e\t)^{2-2\s_0-\s_1}}{2^{1-\s_0-\s_1}} \frac{(\D_0 + \D_1 + \D_2)^{(\D_0 + \D_1 + \D_2) / 2} \D_3^{\D_3/2}}{(2\m)^\m}
\end{align*}
and $F$ is the set of solutions to 
\begin{subequations}
\begin{align}
&\D_0 \geq 3\s_0, \D_1 \geq 2\s_1, \D_2 \geq 1 - \s_0 - \s_1, \D_3 \geq 3\t. \label{constraintsa}\\
&\frac{\D_0 + \D_1 + \D_2}{2} + \s_1 + 2(1 - \s_0 - \s_1) < 2-\e & \mbox{ since $i(S) < (2-\e)|S|$, \qquad see (\ref{good})}.\label{constraintsb}\\ 
& \s_0, \s_1 \geq 0, \s_0 + \s_1 \leq 1.  \label{constraintsc}\\
&0 \leq \t \leq (1 - \e) / \e  \mbox{ since $|S| \geq \e N$}.\label{constraintsd}\\ 
&\m \geq (2 + \e)(1 + \t)  \mbox{ since $M \geq (2 + \e)N$.}\label{constraintse}\\
&\s_0<1, \qquad\text{ otherwise $C_3$ is not connected}.\label{constraintsf}
\end{align}
\end{subequations}
\subsubsection{Eliminating $\m$}
We begin by showing that it is enough to consider $\m = (2 + \e)(1 + \t)$. We collect all terms involving $\m$, including $\D_3, \l$ and $\l_3$ whose values are determined in part by $\m$. It is enough to consider the logarithm of $f$. We have
\begin{multline*}
\frac{\partial \log f}{\partial \m}= 2\log \l + \frac{\partial \l}{\partial \m}\left(\frac{2\m}{\l} - \n\frac{f_2(\l)}{f_3(\l)}\right) + \frac{\partial \l_3}{\partial \m}\left( \t \frac{f_2(\l_3)}{f_3(\l_3)} - \frac{\D_3}{\l_3}\right)  \\
- 2\log \l_3 + \log \D_3 + 1 - \log 2\m - 1
\end{multline*}
by definition of $\l, \l_3$, we have 
$$\frac{2\m}{\l} - \n\frac{f_2(\l)}{f_3(\l)} = 0\text{ and }\frac{\D_3}{\l_3}-\t\frac{f_2(\l_3)}{f_3(\l_3)}= 0,$$ 
and so
\begin{equation}
\frac{\partial \log f}{\partial \m} = 2\log\left(\frac{\l}{\l_3}\right) + \log\left(\frac{\D_3}{2\m}\right)
\end{equation}
We have $\D_3\leq 2\m$ and furthermore, $\l \leq \l_3$ since $g_0$ is an increasing function. Indeed, writing $\iota = i(S) / s \leq 2$, we have $\D_3 + 2\iota = 2\m \geq 4(\t+1)$, so
\begin{equation}\label{lambda3lambda}
g_0(\l_3) - g_0(\l) = \frac{\D_3}{\t} - \frac{2\m}{\n} = \frac{2\m -2 \iota}{\t} - \frac{2\m}{\t+1} = \frac{2\m - 2\iota(\t+1)}{\t(\t+1)} \geq \frac{4 -2 \iota}{\t} \geq 0.
\end{equation}

This shows that $\log f$ is decreasing with respect to $\m$, and in discussing the maximum value of $f$ for $\m\geq (2+\e)(1+\t)$ we may assume that $\m = (2 + \e)(1 + \t)$. 

We now argue that to show that $f \leq \exp\{ -\e^2 / 3\}$ when $\m = (2 + \e)(1 + \t)$, it is enough to show that $f \leq 1$ when $\m = 2(1 + \t)$. Let $2 ( 1 + \t) < \m < (2 + \e)(1 + \t)$. Then by \eqref{D3=} and (\ref{constraintsa})
\begin{eqnarray*}
\D_3 &=& 2\m - 4 - \D_0 - \D_1 - \D_2 + 4\s_0 + 2\s_1 \\
&\leq& 2\m - 4 - 3\s_0 - 2\s_1 - (1 - \s_0 - \s_1) +4\s_0 + 2\s_1\\
&=& 2\m - 5 + 2\s_0 + \s_1 \\
&\leq& 2\m - 2
\end{eqnarray*}
and since $\t \leq 1/\e - 1$ by (\ref{constraintsd}), $\m \leq (2 + \e)(1+\t)$ implies $\m \leq 2 / \e + 1 < 3 / \e$. So,
\begin{equation}
\frac{\partial \log f}{\partial \m} \leq 2\log\left(\frac{\l}{\l_3}\right) + \log\left(\frac{2\m - 2}{2\m}\right)  \leq \log\left(1 - \frac{\e}{3}\right) 
\end{equation}
So, fixing $\textbf{w}' = (\s_0, \s_1, \D_0, \D_1, \D_2, \t)$, let $\m = 2(1 + \t)$ and $\m' = (2 + \e)(1 + \t)$. If $f(\textbf{w}', \m) \leq 1$, then
\begin{equation}
\log f(\textbf{w}', \m') \leq \log f(\textbf{w}', \m) +\e(1 + \t) \log\left(1 - \frac\e3\right)  \leq - \frac{\e^2}{3}.
\end{equation}

This shows that it is enough to prove that $f(\textbf{w}) \leq 1$ for $\textbf{w} \in F'$, defined by
\begin{subequations}
\begin{align}
& \D_0 \geq 3\s_0, \D_1 \geq 2\s_1, \D_2 \geq 1- \s_0 - \s_1, \D_3 \geq 3\t \label{constraint1} \\
& \D_0 + \D_1 + \D_2 \leq 4\s_0 + 2\s_1 \label{constraint2} \\
& \s_0, \s_1 \geq 0, \s_0 + \s_1 \leq 1 \label{constraint3}\\
& 0 \leq \t < \infty \label{constraint4}\\
& \m = 2(1 + \t).\label{constraint5}
\end{align}
\end{subequations}
We have relaxed equation \eqref{constraintsb} to give (\ref{constraint2}) in order to simplify later calculations. In $F'$, $\l$ is defined by 
$$g_0(\l) = \frac{2\m}{\n} = \frac{4(1+\t)}{1+\t}=4,$$ 
so in the remainder of the proof 
$$\l = g_0^{-1}(4) \approx 2.688\text{ is fixed}.$$ 
It will be convenient at times to write $\D = \D_0 + \D_1 + \D_2$. We observe that $3\s_0 + 2\s_1 + (1 - \s_0 -\s_1) = 2\s_0 + \s_1 + 1$, so by (\ref{constraint1}), (\ref{constraint2}),
\begin{equation}\label{eq:dbound}
2\s_0 + \s_1  +1 \leq \D \leq 4\s_0 + 2\s_1.
\end{equation}
Note also that $\m = 2(1 + \t)$ implies
\begin{equation}\label{eq:delta3}
\D_3 = 2\m - 4 - \D_0 - \D_1 - \D_2 + 4\s_0 + 2\s_1 = 4\t + 4\s_0 + 2\s_1- \D.
\end{equation}
The quantity $2\s_0 + \s_1$ will appear frequently. We note that (\ref{eq:dbound}) and $\s_0 + \s_1 \leq 1$ imply
\begin{equation}
1 \leq 2\s_0 + \s_1 \leq 2.
\end{equation}
\subsubsection{Eliminating $\t$}
We now turn to choosing the optimal $\t$. With $\m = 2(1 + \t)$,
\begin{eqnarray}
f(\s_0, \s_1, \D_0, \D_1, \D_2, \t) &=& \frac{(\t+1)^{\t+1}}{\s_0^{\s_0} \s_1^{\s_1} (1 - \s_0 - \s_1)^{1 - \s_0 - \s_1} \t^\t} \left(\frac{\l^4}{f_3(\l)}\right)^{\t+1} \frac{f_3(\l_0)^{\s_0}}{\l_0^{\D_0}} \frac{f_2(\l_1)^{\s_1}}{\l_1^{\D_1}} \nonumber \\
&& \times \frac{f_1(\l_2)^{1 - \s_0 - \s_1}}{\l_2^{\D_2}} \frac{f_3(\l_3)^\t}{\l_3^{\D_3}} \frac{(e \t)^{2 - 2\s_0 - \s_1}}{2^{1 - \s_0 - \s_1}}\times\frac{\D^{\D / 2}\D_3^{\D_3 / 2}}{(4 + 4 \t)^{2 + 2\t}}. \label{eq:fdef} 
\end{eqnarray}
Here $\l_0 = \l_0(\D_0, \s_0)$, $\l_1 = \l_1(\D_1, \s_1), \l_2 = \l_2(\D_2, \s_0, \s_1), \l_3 = \l_3(\s_0, \s_1, \D_0, \D_1, \D_2, \t)$ as defined in \eqref{l0}, \eqref{l1}, \eqref{l2}, \eqref{l3}. Since $\t  f_2(\l_3) / f_3(\l_3) - \D_3 / \l_3 = 0$ by the definition of $\l_3$, the partial derivative of $\log f$ with respect to $\t$ is given by
\begin{eqnarray*}
\frac{\partial}{\partial \t} \log f(\s_0, \s_1, \D_0, \D_1, \D_2, \t) &=& \log (\t + 1) + 1 - \log \t - 1  + \log\left(\frac{\l^4}{f_3(\l)}\right)  \\
&& + \frac{\partial \l_3}{\partial \t} \left(\t \frac{f_2(\l_3)}{f_3(\l_3)} - \frac{\D_3}{\l_3}\right) + \log(f_3(\l_3)) - 4 \log \l_3 \\
&& + \frac{2 - 2 \s_0 - \s_1}{\t} + 2\left(1 + \log \D_3\right) - 2\log(4+4\t) - 2 \\
&=& \log(\t + 1) - \log \t + \log\left(\frac{\l^4}{\l_3^4}\frac{f_3(\l_3)}{f_3(\l)}\right) + \frac{2 - 2\s_0 - \s_1}{\t} \\
&& + 2\log \D_3 - 2\log(4 + 4\t)
\end{eqnarray*}
This is positive for $\t$ close to zero. This is clear as long as $2\s_0+\s_1<2$. But if $2\s_0+\s_1=2$ then $\s_0+\s_1\leq 1$ implies that $\s_0=1,\s_1=0$. But then if $\t>0$ we have that $C_3$ is not connected and that if $\t=0$, $S=[N]$ which violates \eqref{constraintsf}. On the other hand, $\frac{\partial}{\partial \t} \log f$ vanishes if
\begin{equation}\label{eq:tdef}
2 - 2\s_0 - \s_1 - \t\left[\log\left(1 + \frac{1}{\t}\right) -2\log\left(\frac{\D_3}{4\t}\right) - \log\left(\frac{\l^4 f_3(\l_3)}{\l_3^4 f_3(\l)}\right)\right] = 0.
\end{equation}
So any local maximum of $f$ must satisfy this equation. If no solution exists, then it is optimal to let $\t \rightarrow \infty$. We will see below how to choose $\t$ to guarantee maximality. For now, we only assume $\t$ satisfies (\ref{eq:tdef}).
\subsubsection{Eliminating $\D_0,\D_1,\D_2$.}
We now eliminate $\D_0, \D_1, \D_2$. Fix $\s_0, \s_1$. For $\D_i > (3 - i)\s_i$ such that $\D_0 + \D_1 + \D_2 < 4\s_0 + 2\s_1$,
\begin{eqnarray}
\frac{\partial}{\partial \D_i} \log f &=& \frac{\partial \l_i}{\partial \D_i} \brac{\s_i \frac{f_{2-i}(\l_i)}{f_{3-i}(\l_i)} - \frac{\D_i}{\l_i}} - \log \l_i + \log \l_3 \nonumber\\
&& +  \frac{\partial}{\partial \t} \log f\;\;\frac{\partial \t}{\partial \D_i}
 + \frac{1}{2} \log \D + \frac12 - \frac12 \log\D_3 - \frac12\label{oldl}\\
&=&  - \log \l_i + \log\left(\l_3 \sqrt\frac{\D}{\D_3}\right),\nonumber
\end{eqnarray}
since $g_i(\l_i) = \D_i / \s_i$ by definition of $\l_i$, and the term $\frac{\partial}{\partial \t} \log f\;\;\partial \t / \partial \D_i$ vanishes because (\ref{eq:tdef}) is assumed to hold. We note that $\l_i > 0$ when $\D_i > (3 - i)\s_i$ (Appendix A), allowing division by $\l_i$. 

As $\D_i$ tends to its lower bound $(3 - i)\s_i$, we have $\log \l_i \rightarrow -\infty$ while the other terms remain bounded, so the derivative is positive at the lower bound of $\D_i$. Any stationary point must satisfy $\l_0 = \l_1 = \l_2 = \l_3 \sqrt{\D / \D_3} =: \widehat{\l}$. This can only happen if
\begin{equation}
\s_0 g_0(\lhat) + \s_1 g_1(\lhat) + (1 - \s_0 - \s_1) g_2(\lhat) = \s_0\frac{\D_0}{\s_0} + \s_1 \frac{\D_1}{\s_1} + (1-\s_0 - \s_1)\frac{\D_2}{1 - \s_0 - \s_1} = \D.
\end{equation}
So we choose $\lhat$, $\D, \t$ to solve the system of equations
\begin{eqnarray}
\lhat &=& \l_3 \sqrt{\frac{\D}{\D_3}} \nonumber \\
\D &=& \s_0 g_0(\lhat) + \s_1 g_1(\lhat) + (1-\s_0 - \s_1)g_2(\lhat) \label{lamdef} \\
2 - 2\s_0 - \s_1 &=& \t\left[\log\left(1 + \frac{1}{\t}\right) - 2\log\left(\frac{\D_3}{4\t}\right)  - \log\left(\frac{\l^4 f_3(\l_3)}{\l_3^4 f_3(\l)}\right)\right] \nonumber
\end{eqnarray}
In Appendix B we show that this system has no solution such that $2\s_0 + \s_1 + 1 \leq \D \leq 4\s_0 + 2\s_1$ (see (\ref{eq:dbound})). This means that no stationary point exists, and $\log f$ is increasing in each of $\D_0, \D_1, \D_2$. In particular, it is optimal to set 
\beq{setD}{
\D_0 + \D_1 + \D_2 = 4\s_0 + 2\s_1\text{ which implies that $\D_3=4\t$, see \eqref{eq:delta3}}.}
This eliminates one degree of freedom. We now set 
$$\D_2 = 4\s_0 + 2\s_1 - \D_0 - \D_1.$$ 
Then for $\D_0, \D_1$, we have
\begin{equation}
\frac{\partial}{\partial \D_i} \log f = - \log \l_i + \log \l_2, \quad i = 0,1.
\end{equation}
To see this note that \eqref{oldl} has to be modified via the addition of $\frac{\partial}{\partial \D_2} \log f \times \frac{\partial \D_2}{\partial \D_i}$, for $i=0,1$.

So it is optimal to let $\l_0 = \l_1 = \l_2 = \lbar$, defined by
\begin{equation}\label{eq:lbar}
 \s_0 g_0(\lbar) + \s_1 g_1(\lbar) + (1-\s_0 - \s_1)g_2(\lbar) = 4\s_0 + 2\s_1 
\end{equation}
This has a unique solution $\lbar \geq 0$ whenever $2\s_0 + \s_1 \geq 1$, since for fixed $\s_0, \s_1$, the left-hand side is a convex combination of increasing functions, by Lemma \ref{gconvex}, Appendix C. This defines $\D_i = \D_i(\s_0, \s_1)$ by
\begin{equation}\label{deltadef}
\D_0 = g_0(\lbar) \s_0, \quad \D_1 = g_1(\lbar)\s_1, \quad \D_2 = g_2(\lbar)(1-\s_0 - \s_1)
\end{equation}
We note at this point that $\lbar \leq\l$. Indeed, by \eqref{setD} and \eqref{constraintsa},
$$\D_0=4\s_0 + 2\s_1 - \D_1 - \D_2 \leq 4\s_0 + 2\s_1 - 2\s_1 - (1-\s_0 - \s_1) \leq 4\s_0,$$ 
so 
\begin{equation}\label{lbarinequality}
g_0(\lbar) = \frac{\D_0}{\s_0} \leq 4 = g_0(\l)
\end{equation}
implying that $\lbar \leq \l$, since $g_0$ is increasing.

This choice (\ref{deltadef}) of $\D_0, \D_1, \D_2$ simplifies $f$ significantly. With $\D = 4\s_0 + 2\s_1$ we have $\D_3 = 4\t$,  see \eqref{setD}, and so 
\beq{l3=}{
\l_3 = g_0^{-1}\bfrac{4\t}{\t} = \l} 
is fixed. In particular, the relation (\ref{eq:tdef}) for $\t$ simplifies to
\begin{equation}\label{eq:tsimple}
2 - 2\s_0 - \s_1 = \t\log\left(1 + \frac{1}{\t}\right)
\end{equation}
Let $\f(\t) = \t\log(1 + 1/\t)$. Then $\f''(\t) = - \t^{-1}(\t+1)^{-2}$, so $\f$ is concave and then $\lim_{\t\rightarrow 0} \f(\t) = 0$, $\lim_{\t\rightarrow \infty} \f(\t) = 1$ implies that $\f$ is strictly increasing and takes values in $[0, 1)$ for $\t \geq 0$. This means that (\ref{eq:tsimple}) has a unique solution if and only if $2\s_0 + \s_1 > 1$. When $2\s_0 + \s_1 = 1$, $f$ is increasing with respect to $\t$, and we treat this case now.

If $2\s_0 + \s_1 = 1$, then (\ref{eq:dbound}) implies that $\D=2$. Furthermore, $\D_3 = 4\t$ (see (\ref{eq:delta3})) and $\l_3 = \l$ (see \eqref{l3=}) and $g_i(0) = 3 - i$ implies that 
$$\s_0 g_0(0) + \s_1g_1(0) + (1 - \s_0 - \s_1)g_2(0) = 2\s_0 + \s_1 + 1 = 4\s_0 + 2\s_1,$$ 
so $\lbar = 0$ is the unique solution to (\ref{eq:lbar}). Then since $\D_i / \s_i = g_i(0) = 3 - i$ (Lemma \ref{gibounds}, Appendix C), we have $\D_i = (3-i)\s_i$, $i=0,1,2$, and as in (\ref{lambdazero}), \eqref{S1case}, \eqref{S2case},
\begin{equation}
\frac{f_3(\lbar)^{\s_0} f_2(\lbar)^{\s_1} f_1(\lbar)^{1 - \s_0 - \s_1}}{\lbar^\D} = \left(\frac{f_3(\lbar)}{\lbar^3}\right)^{\s_0} \left(\frac{f_2(\lbar)}{\lbar^2}\right)^{\s_1}\left(\frac{f_1(\lbar)}{\lbar}\right)^{1 - \s_0 - \s_1} = \frac{1}{6^{\s_0}}\frac{1}{2^{\s_1}}
\end{equation}
so when $2\s_0 + \s_1 = 1$, (\ref{eq:fdef}) becomes
\begin{equation}
f(\s_0, \s_1, \t) = \frac{(\t+1)^{\t+1}}{\s_0^{\s_0} \s_1^{\s_1} (1 - \s_0 - \s_1)^{1 - \s_0 - \s_1} \t^\t} \frac{\l^4}{f_3(\l)} \frac{1}{6^{\s_0}}\frac{1}{2^{\s_1}} \frac{e\t}{2^{1 - \s_0 - \s_1}} \frac{2^{2 / 2} (4\t)^{2\t}}{(4 + 4\t)^{2 + 2\t}}.
\end{equation}
In this computation we also used the fact that $\l=\l_3$ (see \eqref{l3=}) and $\D_3=4\t$ (see (\ref{eq:delta3})) to find that
$$\left(\frac{\l^4}{f_3(\l)}\right)^{\t+1}\frac{f_3(\l_3)^\t}{\l_3^{\D_3}}=\frac{\l^4}{f_3(\l)}.$$
Here $\l^4 / f_3(\l) \approx 7.05$ is fixed. As noted in the discussion after (\ref{eq:tsimple}), the partial derivative in $\t$ is positive for all $\t$, so we let $\t \rightarrow \infty$. Substituting $\s_1 = 1 - 2\s_0$ we are reduced to
\begin{eqnarray*}
f(\s_0) &=& \lim_{\t \rightarrow \infty} \frac{(\t+1)^{\t+1}}{\s_0^{\s_0} (1 - 2\s_0)^{(1 - 2\s_0)} \s_0^{\s_0} \t^\t} \frac{\l^4}{f_3(\l)} \frac{1}{6^{\s_0}}\frac{1}{2^{1 - 2\s_0}}  \frac{e\t}{2^{\s_0}} \frac{2 (4\t)^{2\t}}{(4 + 4\t)^{2 + 2\t}} \\
&=& \frac{ \l^4}{16 f_3(\l)} \frac{1}{\s_0^{2\s_0} (1 - 2\s_0)^{1 - 2\s_0} 3^{\s_0}}
\end{eqnarray*}
This has the stationary point $\s_0 = 2 - \sqrt{3}$, and $f(2 - \sqrt{3}) \approx 0.95$. We also have $f(0) \approx 0.44$ and $f(1 / 2) \approx 0.51$ at the lower and upper bounds for $\s_0$.

\subsubsection{Dealing with $\s_0,\s_1$}
With this, we have reduced our analysis to the variables $\s_0, \s_1$ in the domain 
$$E = \{(\s_0, \s_1) : \s_0, \s_1 \geq 0, \s_0 + \s_1 \leq 1, 2\s_0 + \s_1 \geq 1\}.$$ 
We just showed that $f \leq 1$ in 
$$E_0 = \{(\s_0, \s_1) \in E : 2\s_0 + \s_1 = 1\}.$$ 
Further define 
\begin{align*}
E_1 &= \{(\s_0, \s_1) \in E : 0.01 \leq \s_1 \leq 0.99\},\\
E_2 &= \{(\s_0, \s_1) \in E : 0 \leq \s_1 < 0.01\},\\ 
E_3 &= \{(\s_0, \s_1) \in E : 0.99 < \s_1 \leq 1\}.
\end{align*}
 We will show that $f \leq 1$ in each of these sets, whose union covers $E$.

From this point on, let $\partial_i = \frac{\partial}{\partial \s_i}, i = 0,1$. As mentioned above, $\D = 4\s_0 + 2\s_1$ simplifies $f$. Specifically, if $2\s_0 + \s_1 > 1$ then \eqref{eq:fdef} becomes, after using \eqref{setD} and \eqref{l3=},
\begin{eqnarray}
f(\s_0, \s_1) &=& \frac{(\t + 1)^{\t + 1}}{\s_0^{\s_0} \s_1^{\s_1} (1 - \s_0 - \s_1)^{1 - \s_0 - \s_1} \t^\t} \frac{\l^4}{f_3(\l)} \frac{f_3(\lbar)^{\s_0} f_2(\lbar)^{\s_1} f_1(\lbar)^{1 - \s_0 - \s_1}}{\lbar^{4\s_0 + 2\s_1}} \nonumber \\
&& \times \frac{(e\t)^{2 - 2\s_0 - \s_1}}{2^{1 - \s_0 - \s_1}} \frac{(4\s_0 + 2\s_1)^{2 \s_0 + \s_1}(4\t)^{2\t}}{(4 + 4\t)^{2 + 2\t}} \label{fdef}
\end{eqnarray}
In (\ref{eq:tsimple}), (\ref{eq:lbar}) respectively, $\t$ and $\lbar$ are given as functions of $\s_0, \s_1$. Recall that $\l = g_0^{-1}(4)$ is constant. So
\begin{align}
&\partial_0 \log f(\s_0, \s_1) =\nonumber\\
& - \log \s_0 - 1 + \log(1-\s_0 - \s_1) + 1 + \log f_3(\lbar) - \log f_1(\lbar)  \nonumber \\
&   - 4\log \lbar -2 \log (e\t) + \log 2 + 2\log(4\s_0 + 2\s_1) + 2 \nonumber \\
& + \frac{\partial \lbar}{\partial \s_0}\left(\s_0 \frac{f_2(\lbar)}{f_3(\lbar)} + \s_1\frac{f_1(\lbar)}{f_2(\lbar)} + (1-\s_0 - \s_1)\frac{f_0(\lbar)}{f_1(\lbar)} - \frac{4\s_0 + 2\s_1}{\lbar}\right) \nonumber \\
& + \frac{\partial \t}{\partial \s_0} \bigg(\log(\t+1) + 1 - \log \t - 1 + \frac{2 - 2\s_0 - \s_1}{\t} + 2\log 4\t + 2 - 2\log(4+4\t) - 2\bigg) \nonumber \\
&= \log\left(\frac{1-\s_0-\s_1}{\s_0}\right) + \log\left(\frac{f_3(\lbar)}{\lbar^4 f_1(\lbar)}\right) - 2\log \t + \log 2 + 2\log (4\s_0 + 2\s_1)\label{eq:dds0}
\end{align}
where, as expected, the terms involving $\partial_0 \t$ and $\partial_0 \lbar$ vanish since $\t, \lbar$ were chosen to maximize $\log f$. (See \eqref{eq:tsimple} and \eqref{eq:lbar} respectively).

Similarly,
\begin{align} 
&\partial_1 \log f(\s_0, \s_1) =\nonumber\\
& - \log \s_1 - 1 + \log(1-\s_0 - \s_1) + 1 + \log f_2(\lbar) - \log f_1(\lbar)  \nonumber \\
&   - 2\log \lbar -\log (e\t) + \log 2 + \log(4\s_0 + 2\s_1) + 1 \nonumber \\
& + \frac{\partial \lbar}{\partial \s_1}\left(\s_0 \frac{f_2(\lbar)}{f_3(\lbar)} + \s_1\frac{f_1(\lbar)}{f_2(\lbar)} + (1-\s_0 - \s_1)\frac{f_0(\lbar)}{f_1(\lbar)} - \frac{4\s_0 + 2\s_1}{\lbar}\right) \nonumber \\
& + \frac{\partial \t}{\partial \s_1} \bigg(\log(\t+1) + 1 - \log \t - 1 + \frac{2 - 2\s_0 - \s_1}{\t} + 2\log 4\t + 2 - 2\log(4+4\t) - 2\bigg) \nonumber \\
& =\log\left(\frac{1 - \s_0-\s_1}{\s_1}\right) + \log\left(\frac{f_2(\lbar)}{\lbar^2 f_1(\lbar)}\right) - \log \t + \log 2 + \log (4\s_0 + 2\s_1). \label{eq:dds1}
\end{align}
Any stationary point must satisfy
\beq{qpr}{
(\partial_0 - 2\partial_1) \log f = \log\left(\frac{\s_1^2}{\s_0(1-\s_0 - \s_1)}\right) + \log\left(\frac{f_1(\lbar)f_3(\lbar)}{f_2(\lbar)^2}\right) - \log 2 = 0.}
It is shown in Lemma \ref{finequalities}, Appendix C that
$$1 \leq \frac{f_2(\lbar)^2}{f_1(\lbar)f_3(\lbar)} \leq 2.$$ 
This means from \eqref{qpr} that if $(\partial_0 - 2\partial_1) \log f=0$ then
$$2 \leq \frac{\s_1^2}{\s_0(1-\s_0 - \s_1)} \leq 4.$$
In particular, the lower bound implies $\s_0 \geq (1 - \s_1) / 2 + \sqrt{1 - 2\s_1 - \s_1^2} / 2$ and the upper bound implies $\s_1 \leq - 2\s_0  + \sqrt{4\s_0 - 4\s_0^2 }$. The latter bound is used only to conclude that $\s_1 < 1 / 2$, by noting that $-2\s_0 + \sqrt{4\s_0 - 4\s_0^2 } \leq (5^{1/2}-1)/3< 1/ 2$ for $0 \leq \s_0 \leq 1$. In conclusion,
\begin{equation}\label{stationaryconditions}
(\partial_0 - 2\partial_1) \log f = 0 \Longrightarrow \left\{\begin{array}{l} \s_0 \geq (1 - \s_1) / 2 + \sqrt{1 - 2\s_1 - \s_1^2}/2. \\ \s_1 < 1 / 2. \end{array}\right. 
\end{equation}

\textbf{Case One.} $E_1 = \{(\s_0, \s_1) \in E : 0.01 \leq \s_1 \leq 0.99\}$

We need a lower bound for $\lbar \t$. We first note that $g_i(\lbar) \leq 3 - i + \lbar$ (Lemma \ref{gibounds}, Appendix C) implies
\begin{equation}
4\s_0 + 2\s_1 = \s_0 g_0(\lbar) + \s_1g_1(\lbar) + (1 - \s_0 - \s_1)g_2(\lbar) \leq 2\s_0 + \s_1 + 1 + \lbar
\end{equation}
so 
$$\lbar \geq 2\s_0 + \s_1 - 1=1-\t\log(1 + 1/\t).$$ 
Here we have used \eqref{eq:tsimple}.

For $\t$, note that $\s_0 < 0.99$ and $\s_0 + \s_1 \leq 1$ implies $\t\log(1 + 1/\t) = 2-2\s_0 - \s_1 \geq 1 - \s_0 > 0.01$. The function $\t\log(1 + 1 / \t)$ is increasing in $\t$ by the discussion after (\ref{eq:tsimple}). This implies 
\beq{lowtau}{
\t > 10^{-3},} 
since $0.001 \log(1001) < 0.01$. 

If $\t\leq 1.1$,
$$\lbar\geq 1-1.1\log \left(1 + \frac{1}{1.1}\right) >0.1.$$
So, if $\t\leq 1.1$,
$$\lbar\t\geq 10^{-4}.$$
If $1.1<\t$ then we use $\log(1+x)\leq x-x^2/2+x^3/3$ for $|x|\leq 1$ to write
$$\lbar\t\geq \t-\t^2\log(1+1/\t)\geq \frac{1}{2}-\frac{1}{3\t}\geq \frac{1}{6}.$$
So, in $E_1$, we have 
\begin{equation}\label{ltbound}
\lbar \t \geq 10^{-4}.
\end{equation}

By definition of $E_1$, $\s_0 \geq 0.01$ and $\s_1 \geq 0.01$. By (\ref{lbarinequality}), $0 \leq \lbar \leq \l$. This implies $f_3(\lbar) / \lbar^2 f_1(\lbar) \leq 1/6$ and $f_2(\lbar) / \lbar f_1(\lbar) \leq 1/3$ (Lemma \ref{finequalities}, Appendix C). So after rewriting (\ref{eq:dds0}) slightly,
\begin{eqnarray}
\partial_0 \log f(\s_0, \s_1) &=&\log\left(\frac{1 - \s_0 - \s_1}{\s_0}\right) + \log\left(\frac{f_3(\lbar)}{\lbar^2f_1(\lbar)}\right) - 2\log \lbar\t + \log 2 + 2\log(4\s_0 + 2\s_1) \nonumber \\
&\leq& \log \frac{1}{0.01} + \log \frac{1}{6} -2\log 10^{-4} + \log 2 + 2\log 4 \label{d0bound} \\
&\leq& 25. \nonumber
\end{eqnarray}
Similarly, (\ref{eq:dds1}) is bounded by
\begin{eqnarray*}
\partial_1 \log f(\s_0, \s_1) &\leq& \log \frac{1}{0.01} + \log \frac13 - \log 10^{-4} + \log 2 + \log 4 \leq 15.
\end{eqnarray*}
We now show numerically that $\log f \leq 0$ in $E_1$.

{\bf Numerics of Case One:}\\
Since $\partial_i \log f$ is only bounded from above, $i =0, 1$, this requires some care at the lower bounds of $\s_0, \s_1$, given by $\s_0 \geq (1 - \s_1) / 2$ and $\s_1 \geq 0.01$. Note that if $\s_0 = (1 - \s_1) / 2$, then $(\s_0, \s_1) \in E_0$ and it was shown above that $\log f(\s_0, \s_1) \leq \log 0.96 \leq -0.01$. Define a finite grid $P \subseteq E_1$ such that for any $(\s_0, \s_1) \in E_1$, there exists $(\overline{\s}_0, \overline{\s}_1) \in P \cup E_0$ where $0 \leq \s_0 - \overline{\s}_0 \leq \d$ and $0 \leq \s_1 - \overline{\s}_1 \leq \d$. Here $\d = 1 / 4000$. Numerical calculations will show that $\log f(\overline{\s}_0, \overline{\s}_1) \leq -0.01$ for all $(\overline{\s}_0, \overline{\s}_1) \in P$. This implies that for all $\s_0, \s_1 \in E_1$,
$$\log f(\s_0, \s_1) \leq \max_{\overline{\s}_0, \overline{\s}_1 \in P \cup E_0} \log f(\overline{\s}_0, \overline{\s}_1) + 25\d + 15\d \leq -0.01 + 40\d \leq 0.$$

When calculating $\log f(\overline{\s}_0, \overline{\s}_1)$, approximations $\lbar_{num}, \t_{num}$ of $\lbar(\overline{\s}_0, \overline{\s}_1), \t(\overline{\s}_0, \overline{\s}_1)$ must be calculated with sufficient precision. By definition of $\lbar$, $\partial  \log f / \partial \lbar = 0$, while
\begin{align*}
&\left|\frac{\partial^2 \log f}{\partial \lbar^2}\right|\\
& = \bigg| \s_0\left(\frac{f_1(\lbar)}{f_3(\lbar)} - \frac{f_2(\lbar)^2}{f_3(\lbar)^2}\right) + \s_1 \left(\frac{f_0(\lbar)}{f_2(\lbar)} - \frac{f_1(\lbar)^2}{f_2(\lbar)^2}\right) \\
& \hspace{1in}+ (1-\s_0 - \s_1)\left(\frac{f_0(\lbar)}{f_1(\lbar)} - \frac{f_0(\lbar)^2}{f_1(\lbar)^2}\right) + \frac{4\s_0 + 2\s_1}{\lbar^2} \bigg|\\
&=\frac{1}{\lbar^2}  \bigg| \s_0\left(\lbar^2\frac{f_1(\lbar)}{f_3(\lbar)} - \frac{\lbar^2f_2(\lbar)^2}{f_3(\lbar)^2}\right) + \s_1 \left(\lbar^2\frac{f_0(\lbar)}{f_2(\lbar)} - \frac{\lbar^2f_1(\lbar)^2}{f_2(\lbar)^2}\right) \\
& + (1-\s_0 - \s_1)\left(\lbar^2\frac{f_0(\lbar)}{f_1(\lbar)} - \frac{\lbar^2f_0(\lbar)^2}{f_1(\lbar)^2}\right) + \s_0g_0(\lbar) + \s_1g_1(\lbar) + (1-\s_0-\s_1)g_2(\lbar) \bigg| \\
&= \frac{1}{\lbar^2}| \s_0g_0(\lbar)(g_1(\lbar) - g_0(\lbar) + 1) + \s_1g_1(\lbar)(g_2(\lbar) - g_1(\lbar) + 1) \\
& \hspace{2in}+ (1-\s_0-\s_1)g_2(\lbar)(\lbar - g_2(\lbar) + 1)| \\
&\leq \frac{9}{\lbar^2} |\s_0g_0(\lbar) + \s_1g_1(\lbar) + (1-\s_0 -\s_1)g_2(\lbar)| \\
&=\frac{9}{\lbar^2}|4\s_0+2\s-1|,\qquad\text{by \eqref{eq:lbar}}\\
&\leq \frac{36}{\lbar^2}.
\end{align*}
Here we use the fact that $g_i(\lbar) \leq 4$ for $0\leq \lbar \leq \l$, $i = 0,1,2$ to conclude that $|g_1-g_0+1|, |g_2-g_1+1|, |\lbar - g_2+1| \leq 9$, and the final step uses $4\s_0 + 2\s_1 \leq 4$. So the error contributed by $\lbar_{num}$ is
\begin{equation}
|\log f(\overline{\s}_0, \overline{\s}_1; \lbar_{num}) - \log f(\overline{\s}_0, \overline{\s}_1; \lbar)| \leq (\lbar_{num}-\lbar)^2 \frac{36}{\lbar^2}
\end{equation}
and to achieve a numerical error of at most $10^{-4}$, we require that $|\lbar_{num} / \lbar - 1| \leq 10^{-2} / 6$.

Similarly by definition of $\t$, $\partial  \log f / \partial \t = 0$, while 
$$\left|\frac{\partial^2\log f}{\partial \t^2}\right|=\left|\frac{1}{\t(\t+1)}\right|
\leq 10^3,\text{ by \eqref{lowtau}}.$$
Thus to achieve a numerical error of at most $10^{-4}$, it suffices to have $|\t_{num} / \t - 1| \leq 10^{-2}$. 

With the above precision, it is found that over all $(\overline{\s}_0, \overline{\s}_1) \in P \cup E_0$, $\log f(\overline{\s}_0, \overline{\s}_1) \leq -0.0105$ numerically. With an error tolerance of $10^{-4}$, this shows that $\log f(\overline{\s}_0, \overline{\s}_1) \leq -0.01$.

\textbf{Case Two.} $E_2 = \{(\s_0, \s_1) \in E : 0 \leq \s_1 < 0.01\}$

We divide $E_2$ into three subregions,
\begin{align*}
&E_{2, 1} = \{(\s_0, \s_1) \in E_2 : \s_1 = 0\},\\
&E_{2, 2} = \{(\s_0, \s_1) \in E_2 : \s_0 + \s_1 = 1\},\\ 
&E_{2, 3} = E_2 \setminus (E_{2, 1} \cup E_{2, 2}).
\end{align*}
We begin by considering the point $(\s_0, \s_1) = (1, 0)$. Here $4\s_0 + 2\s_1 = 4$, and from (\ref{eq:lbar}) $\lbar$ is defined by $g_0(\lbar) = 4$. So $\lbar = g_0^{-1}(4) = \l$. We also have $2 - 2\s_0 - \s_1 = 0$, and from the definition (\ref{eq:tsimple}) of $\t$ we have $\t = 0$. Plugging this into the definition of $f$ (\ref{fdef}) gives $f(1, 0) = 1$.

{\bf Sub-Case 2.1a:}\\
Now consider $E_{2, 1}$, where $\s_1 = 0$. Here $\s_0\geq 1/2$, from the definition of $E$ and
\begin{equation*}
\partial_0 \log f(\s_0, 0) = \log\left(\frac{1 - \s_0}{\s_0}\right) + \log\left(\frac{f_3(\lbar)}{\lbar^2f_1(\lbar)}\right) - 2\log \lbar\t + \log 2 + 2\log(4\s_0)
\end{equation*}
Within $E_{2, 1}$, we consider two cases. First suppose $\s_0 \leq 0.99$. As noted in (\ref{ltbound}), $\s_0 \leq 0.99$ implies $\lbar \t \geq 10^{-4}$. Applying the same bounds as in (\ref{d0bound}),
\begin{equation}
\partial_0 \log f(\s_0, 0) \leq  \log \frac{1}{6} - 2 \log 10^{-4} + \log 2 + 2\log 4 \leq 21
\end{equation}
and we show numerically that $f \leq 1$. The numerical calculations for this case now follow the same outline as above. The precision requirements given there will suffice in this case. 

{\bf Sub-Case 2.1b:}\\
Now suppose $\s_0 \geq 0.99$, still assuming $\s_1 = 0$. Here $\lbar \leq \l$ (see (\ref{lbarinequality})) implies $f_3(\lbar) / \lbar^4 f_1(\lbar) \geq 0.01$ by Lemma \ref{finequalities}, Appendix C. We have $\t\log(1 + 1 / \t) = 2 - 2\s_0 - \s_1 = 2 - 2\s_0 \leq 0.02$ and since $\t\log(1 + 1 / \t)$ is increasing (see (\ref{eq:tsimple})), it follows from a numerical calculation that $\t \leq 0.004$. This implies
\begin{equation}
\frac{1 - \s_0}{\t^2} = \frac{\log\left(1 + \frac{1}{\t}\right)}{2\t} \geq 125\log 250
\end{equation}
and
\begin{eqnarray*}
\partial_0 \log f(\s_0, 0) &=& \log\left(\frac{1 - \s_0 }{\s_0}\right) + \log\left(\frac{f_3(\lbar)}{\lbar^4 f_1(\lbar)}\right) - 2\log \t + \log 2 + 4\log (4\s_0) \\
&=& \log\left(\frac{1-\s_0}{\t^2}\right) - \log \s_0 + \log\left(\frac{f_3(\lbar)}{\lbar^4 f_1(\lbar)}\right) + \log 2 + 4\log (4\s_0) \\
&\geq& \log (125\log 250) + \log 0.01 + \log 2 + 2\log 3.96 > 0
\end{eqnarray*}
which implies $f(\s_0, 0) < f(1, 0) = 1$ for $\s_0 \geq 0.99$.

{\bf Sub-Case 2.2:}\\
Now consider $E_{2, 2}$, i.e. suppose $\s_0 + \s_1 = 1$ and $\s_1 < 0.01$. Then
\begin{equation}\label{79}
\partial_0 \log f(\s_0, 1 - \s_0) = \log\left(\frac{1 - \s_0}{\s_0}\right) + \log\left(\frac{f_3(\lbar)}{\lbar^2 f_2(\lbar)}\right) - \log \t + \log(2 + 2\s_0)
\end{equation}

By Lemma \ref{finequalities}, Appendix C, $\lbar \leq \l$ implies 
$$\frac{f_3(\lbar)}{ \lbar^2f_2(\lbar)} > 0.09.$$

As $\s_1 = 1 - \s_0$, $\t$ is defined by $\t \log(1  + 1/\t)  = 2 - 2\s_0 - \s_1 =  \s_1$. So $\t\log(1 + 1 / \t) \leq 0.01$, implying $\t \leq 0.003$ since $\t\log(1 + 1/\t)$ is increasing, and so
\begin{equation}
\frac{1 - \s_0}{\t} = \frac{\s_1}{\t} = \log\left(1 + \frac{1}{\t}\right) > \log 333.
\end{equation}
So,
\begin{eqnarray*}
\partial_0 \log f(\s_0, 1 - \s_0) &=& \log\left(\frac{1-\s_0}{\t}\right) + \log\left(\frac{f_3(\lbar)}{\lbar^2 f_2(\lbar)}\right) - \log \s_0 + \log(2 + 2\s_0) \\
&\geq& \log \log 333 + \log 0.09  + \log 3.98 \\
&>& 0 
\end{eqnarray*}
and for all $0.99\leq \s_0<1$, $f(\s_0, 1 - \s_0) < f(1, 0) = 1$.

{\bf Sub-Case 2.3:}\\
Now consider $E_{2, 3}$, i.e. suppose $0 < \s_1 < 1 - \s_0$ and $\s_1 < 0.01$. We show that the gradient $\nabla \log f \neq 0$. Assume $(\partial_0 - 2\partial_1)\log f = 0$. By (\ref{stationaryconditions}) we must have $\s_0 \geq (1 - \s_1) / 2 + \sqrt{1 - 2\s_1 - \s_1^2} / 2$. Since $\s_1 \leq 0.01$, we can replace this by the weaker bound $\s_0 \geq 1 - 1.1\s_1 $. 
We trivially have $1 - \s_0 \geq (2-2\s_0 - \s_1) / 2$, so
\begin{equation}
\frac{\s_1}{\t} \geq \frac{1}{1.1} \frac{1 - \s_0}{\t} \geq \frac{1}{2.2} \frac{2-2\s_0 -\s_1}{\t} = \frac{1}{2.2}\log\left(1 + \frac{1}{\t}\right)
\end{equation}
Since $\t\log(1 + 1/\t) = 2 - 2\s_0 - \s_1 \leq 1.2\s_1  \leq 0.012$, we have $\t < 0.002$. So $\s_1 / \t \geq \log(500) / 2.2$.

This allows us to show that if $(\partial_0 - 2\partial_1) \log f = 0$ and $\s_1 \leq 0.01$, then $(\partial_0 - \partial_1) \log f \neq 0$. Noting that $4\s_0 + 2\s_1 \geq 4(1 - 1.1\s_1) + 2\s_1 \geq 3.976$,
\begin{eqnarray*}
(\partial_0 - \partial_1)\log f &=& \log\left(\frac{\s_1}{\t}\right) + \log\left(\frac{f_3(\lbar)}{\lbar^2f_2(\lbar)}\right) - \log \s_0 + \log(4\s_0 + 2\s_1)  \\
&\geq&  \log (\log(500)/2.2) + \log 0.09 + \log 3.976 \\
&=&1.038445...-2.407945...+1.380276...\\
&>& 0
\end{eqnarray*}
This shows that $\nabla \log f \neq 0$ in $E_{2, 3}$. The boundary of $E_{2, 3}$ is contained in $E_0 \cup E_{2, 1} \cup E_{2, 2} \cup E_1$. Since $f \leq 1$ on the boundary of $E_{2, 3}$ and $\nabla \log f \neq 0$ in $E_{2, 3}$, it follows that $f\leq 1$ in $E_{2, 3}$.

\textbf{Case Three:} $E_3 = \{(\s_0, \s_1) \in E :  0.99 < \s_1 \leq 1\}$.

Further divide $E_3$ into
\begin{align*}
E_{3, 1}& = \{(\s_0, \s_1) \in E_3 : \s_0 + \s_1 = 1\},\\ 
E_{3, 2}& = E_3 \setminus E_{3, 1}.
\end{align*}

{\bf Sub-Case 3.1:}\\
Consider $E_{3, 1}$, i.e. suppose $\s_0 + \s_1  = 1$ and $\s_0 < 0.01$. Then we write, see \eqref{79},
\begin{equation}
\partial_0 \log f(\s_0, 1 - \s_0) = \log\left(\frac{1 - \s_0}{\s_0}\right) + \log\left(\frac{1}{g_0(\lbar)}\right) - \log \lbar\t + \log(2 + 2\s_0)
\end{equation}
To show that this is positive, we bound $\lbar \t$ from above. From (\ref{tau2}) (Appendix B) with $\D = 4\s_0 + 2\s_1$ we have $\t \leq 1 / (4\s_0 + 2\s_1 -2 )$. For $\lbar$, we use the bound derived in Appendix B (\ref{l2bound}). Note that if $\D=4\s_0+2\s_1$ then $L_2=\lbar$ in (\ref{l2bound}). So, 
\begin{equation}
\lbar\leq \frac{12(4\s_0 + 2\s_1 - 2\s_0 - \s_1 - 1)}{6 - 3\s_0 - 2\s_1} \leq 12(2\s_0 + \s_1 - 1)\le 12.
\end{equation}
These two bounds together imply $\lbar \t \leq 6$. For all $0 \leq \lbar \leq \l$ we have $3 \leq g_0(\lbar) \leq 4$ since $3 \leq \D_0 / \s_0 \leq 4$ (see the discussion before (\ref{lbarinequality})).

We conclude that
\begin{equation}
\partial_0 \log f(\s_0, 1 - \s_0) \geq \log \frac{0.99}{0.01} + \log \frac{1}{4} - \log 6 + \log 2 > 0
\end{equation}
This implies that for all $(\s_0, \s_1) \in E_{3, 1}$, $f(\s_0, \s_1) \leq f(0.01, 0.99) \leq 1$, since $(0.01, 0.99) \in E_1$.

{\bf Sub-Case 3.2:}\\
Now consider $E_{3, 2}$. As noted in (\ref{stationaryconditions}), any stationary point of $\log f$ must satisfy $\s_1 < 1 / 2$, so $E_{3, 2}$ contains no stationary point. The boundary of $E_{3, 2}$ is contained in $E_0 \cup E_1 \cup E_{3, 1}$, and it has been shown that $f \leq 1$ in each of $E_0, E_1, E_{3, 1}$. It follows that $f \leq 1$ in $E_{3, 2}$.

This completes the proof of Lemma \ref{supercritical} and Theorem \ref{th1b}.

\section*{Appendix B}

This section is concerned with showing that the system of equations (\ref{lamdef}) under certain conditions has no solution. Throughout the section, assume $\t$ satisfies (\ref{eq:tdef}): Recall that $\D_3=4\t + 4\s_0 + 2\s_1 - \D$,
\begin{equation}\label{tcomplicated}
\t\left(\log\left( 1 + \frac{1}{\t}\right) - 2\log\left(\frac{4\t + 4\s_0 + 2\s_1 - \D}{4\t}\right) - \log\left(\frac{\l^4}{\l_3^4}\frac{f_3(\l_3)}{f_3(\l)}\right)\right) = 2 - 2\s_0 - \s_1.
\end{equation}
Here $\l = g_0^{-1}(4) \approx 2.688$ is fixed, and $\l_3$ is defined by $\l_3 = g_0^{-1}(\D_3 / \t)$.

Define for $2\s_0 + \s_1 + 1 \leq \D \leq 4\s_0 + 2\s_1$
\begin{equation}
L_1(\s_0, \s_1, \D, \t) = \l_3 \sqrt{\frac{\D}{4\t + 4\s_0 + 2\s_1 - \D}}
\end{equation}
and define $L_2(\s_0, \s_1, \D)$ as the unique solution to $G(\s_0, \s_1, L_2(\s_0, \s_1, \D)) = \D$, where $G$ is defined by
\begin{equation}\label{l2def}
G(\s_0, \s_1,x) = \s_0 g_0(x) + \s_1 g_1(x) + (1 - \s_0 - \s_1) g_2(x).
\end{equation}
This is well defined because each $g_i$ is strictly increasing, and for fixed $\s_0, \s_1$ we have $G(\s_0, \s_1, 0) = 2\s_0 + \s_1 + 1 \leq \D$ and $\lim_{x \rightarrow \infty} G(\s_0, \s_1, x) = \infty$ (see Appendix B). Define 
$$R=\set{(\s_0, \s_1, \D, \t)\in \mathbb{R}_+^4: \s_0 + \s_1 \leq 1; 2\s_0 + \s_1 \geq 1;\; 2\s_0 +\s_1 + 1 \leq \D \leq 4\s_0 + 2\s_1;\;\text{(\ref{tcomplicated}) holds}.}$$
We prove that the system (\ref{lamdef}) is inconsistent by proving
\begin{lemma}
Let $(\s_0, \s_1, \D, \t) \in R$. Then $L_1(\s_0, \s_1, \D, \t) > L_2(\s_0, \s_1, \D)$
\end{lemma}
\begin{proof}
Define $L(\s_0, \s_1, \D, \t) = L_1(\s_0, \s_1, \D, \t) - L_2(\s_0, \s_1, \D)$. We will bound $|\nabla L|$ in $R$ in order to show numerically that $L > 0$. However, $\nabla L$ is unbounded for $\D$ close to $4$ and $2\s_0 + \s_1$ close to $1$. For this reason, define 
\begin{align*}
R_1& = \{(\s_0, \s_1, \D, \t) \in R : \D \geq 3.6\},\\ 
R_2& = \{(\s_0, \s_1, \D, \t) \in R : 2\s_0 + \s_1 \leq 1.1\},\\ 
R_3& = R \setminus (R_1 \cup R_2).
\end{align*}
 Analytical proofs will be provided for $R_1, R_2$, and a numerical calculation will have to suffice for $R_3$.

First note that for any $\s_0, \s_1$ we have $L_2(\s_0, \s_1, 2\s_0 + \s_1 + 1) = 0$, since $G(\s_0, \s_1, 0) = 2\s_0 + \s_1 + 1$, see (\ref{l2def}). Here we use $g_i(0) = 3 - i$, $i = 0, 1, 2$. This implies that $L_1(\s_0, \s_1, 2\s_0 + \s_1 + 1, \t) > 0 = L_2(\s_0, \s_1, 2\s_0 + \s_1 + 1)$, and we may therefore assume $\D > 2\s_0 + \s_1 + 1$.

We proceed by finding an upper bound for $\t$, given that it satisfies (\ref{tcomplicated}). Fix $\s_0, \s_1, \D$ and define
\begin{equation}
r(\z) = \z\left(\log\left(1 + \frac{1}{\z}\right) - 2\log\left(\frac{4\t + 4\s_0 + 2\s_1 - \D}{4\t}\right) - \log\left(\frac{\l^4}{\l_3^4}\frac{f_3(\l_3)}{f_3(\l)}\right)\right)
\end{equation}
We first derive a lower bound $r_1(\z) \leq r(\z)$.

For $x \geq 0$ we have $x - x^2 / 2 \leq \log(1 + x) \leq x$. This implies, that for all $\z$,
\begin{equation}
  2\z \log\left(1 + \frac{4\s_0 + 2\s_1 - \D}{4\z}\right) \leq  2\z \frac{4\s_0 + 2\s_1 - \D}{4\z} = \frac{4\s_0 + 2\s_1 - \D }{2}
\end{equation}

Let $h(x) = \log f_3(x) - 4 \log x $. Then $h'(x) = f_2(x) / f_3(x) - 4 / x$, and we note that $h'(\l) = 0$, by definition of $\l$. The second derivative is $h''(x) = f_1(x) / f_3(x) - f_2(x)^2 / f_3(x)^2  +  4 / x^2$. Substituting $f_1(x) = f_3(x) + x + x^2 / 2$ and $f_2(x) = f_3(x) + x^2 / 2$, for all $x \geq \l$
\begin{eqnarray*}
h''(x) &=& \frac{4}{x^2}  + 1 + \frac{x + x^2 / 2}{f_3(x)}  - 1 - \frac{x^2}{f_3(x)} - \frac{x^4}{4f_3(x)^2}\\
&=& \frac{4}{x^2} - \frac{x^2 - 2x}{2f_3(x)} - \frac{x^4}{4f_3(x)^2} \\
\end{eqnarray*}
Since $x \geq \l > 2$ we have $x^2 - 2x > 0$, and $f_3(x) < e^x$ implies
\begin{eqnarray*}
h''(x) &=& \frac{4}{x^2} - \frac{x^2 - 2x}{2f_3(x)} - \frac{x^4}{4f_3(x)^2} \\
&\leq& \frac{4}{x^2} - \frac{x^2 - 2x}{2e^x} \\
&\leq& \frac{4}{x^2} + \frac{2x}{2e^x} \\
&\leq& \frac{4}{x^2} + x^{1 - \l}
\end{eqnarray*}
Here we use the fact that $e^x \geq x^{\l}$ for $x \geq \l$, since $\l<e$. Since $4x^{-2} + x^{1 - \l}$ is decreasing, we have $h''(x) \leq 4\l^{-2} + \l^{1 - \l} <  3 / 4$ for all $x \geq \l$.

By Taylor's theorem, for some $x \in [\l, \l_3]$
\begin{eqnarray*}
\log\left(\frac{\l^4}{\l_3^4}\frac{f_3(\l_3)}{f_3(\l)}\right) &=& h(\l_3) - h(\l) \\
&=& h(\l) + h'(\l)(\l_3 - \l) + \frac{1}{2}h''(x)(\l_3 - \l)^2 - h(\l) \\
&\leq& \frac{3}{8}(\l_3 - \l)^2
\end{eqnarray*}
Another application of Taylor's theorem lets us bound 
$$\l_3 - \l = g_0^{-1}\brac{4 + \frac{4\s_0 + 2\s_1 - \D}{\t}} - g_0^{-1}(4).$$ 
By Lemma \ref{gconvex}, Appendix B, we have $g_0'(x) \geq g_0'(\l) \geq 1/ 2$ for $x \geq \l$, so $dg_0^{-1}(y)/dy \leq 2$ for $y \geq 4$, and for some $y \geq 4$
\begin{equation}
\l_3  = \l + \frac{dg_0^{-1}(y)}{dy} \left( \frac{4\s_0 + 2\s_1 - \D}{\t} \right) \leq \l + 2 \frac{4\s_0 + 2\s_1 - \D}{\t}
\end{equation}
and so
\begin{equation}
\log\left(\frac{\l^4}{\l_3^4}\frac{f_3(\l_3)}{f_3(\l)}\right) \leq \frac{3}{8} (\l_3 - \l)^2 \leq \frac{3}{2} \left(\frac{4\s_0 + 2\s_1 - \D}{\t}\right)^2
\end{equation}

Define $\t_1$ as the unique solution $\z$ to
$$2 - 2\s_0 - \s_1 = r_1(\z)$$
where 
$$r_1(\z)= \z\left(\log\left(1 + \frac{1}{\z}\right) - \frac{4\s_0 + 2\s_1 - \D}{2\z}  -\frac{3}{2}\left(\frac{4\s_0 + 2\s_1 - \D}{\z}\right)^2\right).$$
Then $r_1(\z) \leq r(\z)$, and $r_1(\z)$ is strictly increasing by the discussion after (\ref{eq:tsimple}). So, since $r_1(\t_1) = r(\t) = 2-2\s_0 - \s_1$, it follows that $\t \leq \t_1$.

{\bf Case of $R_1$:}\\
Now fix $(\s_0, \s_1, \D, \t) \in R_1$, i.e. suppose $\D \geq 3.6$. Then
\begin{eqnarray*}
r_1\left( \frac{3}{4}\right) &=& \frac{3}{4}\log \left(1 + \frac{4}{3}\right) - \frac{4\s_0 + 2\s_1 - \D}{2} - 2(4\s_0 + 2\s_1 - \D)^2 \\
&=& \frac{3}{4}\log \frac73 - 2\s_0 - \s_1 + \frac{\D}{2} - 2(4\s_0 + 2\s_1 - \D)^2 \\
&\geq& \frac{3}{4}\log \frac73 - 2\s_0 - \s_1 + \frac{3.6}{2} - 2(4 - 3.6)^2 \\
&>& 2 - 2 \s_0 - \s_1
\end{eqnarray*}
We have $\lim_{\z\rightarrow 0} r_1(\z) \leq 0$, and $r_1$ is continous and increasing, so $\t \leq \t_1 < 3/4$. Since $\D \geq 3.6$ and $2\s_0 + \s_1 \leq 2$,
\begin{equation}
\D - (4\t + 4\s_0 + 2\s_1 - \D) \geq 2\D - 3 - 4\s_0 - 2\s_1  \geq 7.2-7 > 0
\end{equation}
This implies that
\begin{equation}
L_1(\s_0, \s_1, \D) = \l_3\sqrt{\frac{\D}{4\t + 4\s_0 + 2\s_1 - \D}} > \l_3
\end{equation}
Note that 
$$G(\s_0,\s_1,\l)\geq G(\s_0,\s_1,\lbar)=4\s_0+2\s_1\geq \D$$
implies that  
$$L_2(\s_0, \s_1, \D) \leq \l = g_0^{-1}(4).$$ 
Also note that by \eqref{l3} and \eqref{eq:delta3} we have
$$\l_3 = g_0^{-1}\bfrac{\D_3}{\t}=g_0^{-1}\brac{4 + \frac{4\s_0 + 2\s_1 - \D}{\t}} \geq g_0^{-1}(4) = \l,$$ 
since $g_0^{-1}$ is increasing (Lemma \ref{gconvex}, Appendix B). So 
$$L_1(\s_0, \s_1, \D, \t)>\l_3 \geq \l \geq L_2(\s_0, \s_1, \D)$$ 
for $(\s_0, \s_1, \D, \t) \in R_1$.

{\bf Case of $R_2,R_3$:}\\
For $R_2, R_3$ we will need a new bound on $\t$. Since $x - x^2 / 2 \leq \log(1 + x)$ for all $x\geq 0$,
$$r_1(\z) \geq r_2(\z)=\z\left( \frac{1}{\z} - \frac{1}{2\z^2} - \frac{4\s_0 + 2\s_1 - \D}{2\z} - \frac{3}{2}\left(\frac{4\s_0 + 2\s_1 - \D}{\z}\right)^2\right).$$
Let $\t_2$ be defined by $r_2(\t_2) = 2-2\s_0 - \s_1$, which can be solved for $\t_2$;
$$\t_2 = \frac{1 + 3(4\s_0 + 2\s_1 - \D)^2}{\D - 2}.$$
It follows from $r(\t) \geq r_2(\t)$ and the fact that $r_2$ is increasing that 
\begin{equation}\label{tau2}
\t \leq \frac{1 + 3(4\s_0 + 2\s_1 - \D)^2}{\D - 2}.
\end{equation}
An upper bound for $L_2(\s_0, \s_1, \D)$ will follow from bounding the partial derivative of $G(\s_0, \s_1, x)$ with respect to $x$. We have $g_0' \geq 1 / 4$, $g_1' \geq 1 / 3$ and $g_2' \geq 1 / 2$ by Lemma \ref{gconvex} (Appendix B), so
\begin{eqnarray*}
\frac{\partial}{\partial x}G(\s_0, \s_1, x) &=& \s_0 g_0'(x) + \s_1 g_1'(x) + (1 - \s_0 - \s_1) g_2'(x) \\
&\geq& \frac{\s_0}{4} + \frac{\s_1}{3} + \frac{1 - \s_0 - \s_1}{2} \\
&=& \frac{6 - 3\s_0 - 2\s_1}{12}
\end{eqnarray*}
and $G(\s_0, \s_1, 0) = 2\s_0 + \s_1 + 1$ implies
\begin{eqnarray*}
\D &=& G(\s_0, \s_1, L_2(\D)) \\
&\geq& G(\s_0, \s_1, 0) + \min_x \frac{\partial}{\partial x}G(\s_0, \s_1, x) L_2(\D) \\
&\geq& 2\s_0 + \s_1 + 1 + \frac{6 - 3\s_0 - 2\s_1}{12} L_2(\D)
\end{eqnarray*}
So
\begin{equation}\label{l2bound}
L_2(\D) \leq \frac{12(\D - 2\s_0 - \s_1 - 1)}{6 - 3\s_0 - 2\s_1}.
\end{equation}
So, to show $L_1(\s_0, \s_1, \D, \t) \geq L_2(\s_0, \s_1, \D)$, it is enough to show that
\begin{equation}
\l_3 \sqrt{\frac{\D}{4\t + 4\s_0 + 2\s_1 - \D}} > \frac{12(\D - 2\s_0 - \s_1 - 1)}{6 - 3\s_0 - 2\s_1}
\end{equation}
Solving for $\t$, this is equivalent to showing
\begin{equation}
\t < \D\left[\frac{\l_3(6-3\s_0 - 2\s_1)}{24(\D - 2\s_0 - \s_1 - 1)}\right]^2 - \frac{4\s_0 + 2\s_1 - \D}{4}
\end{equation}
and by (\ref{tau2}), and $\l_3 \geq \l$, it is enough to show
\begin{equation}\label{l1l2}
 \frac{1 + 3(4\s_0 + 2\s_1 - \D)^2}{\D - 2} < \D\left[\frac{\l(6-3\s_0 - 2\s_1)}{24(\D - 2\s_0 - \s_1 - 1)}\right]^2 - \frac{(4\s_0 + 2\s_1 - \D)}{4}
\end{equation}
for $(\s_0, \s_1, \D, \t) \in R_2 \cup R_3$.

{\bf Case of $R_2$:}\\
Consider $R_2$, i.e. suppose $2\s_0 + \s_1 \leq 1.1$. Then $4\s_0 + 2\s_1 - \D \leq 2\s_0 + \s_1 - 1 \leq 0.1$ since $\D \geq 2\s_0 + \s_1 + 1$.  This implies
\begin{equation}
\frac{1 + 3(4\s_0 + 2\s_1 - \D)^2}{\D - 2} \leq \frac{1.03}{\D - 2}
\end{equation}
Furthermore, $6-3\s_0 - 2\s_1 \geq 4.9 - \s_0 - \s_1 \geq 3.9$, while $2\s_0 + \s_1 \geq 1$ implies $\D - 2\s_0 - \s_1 - 1 \leq \D - 2$. We have $\l > 2.5$, so it holds that
\begin{equation}
 \D\left[\frac{\l(6-3\s_0 - 2\s_1)}{24(\D - 2\s_0 - \s_1 - 1)}\right]^2 - \frac{(4\s_0 + 2\s_1 - \D)}{4} > \D\left[\frac{2.5 \times 3.9}{24(\D - 2)}\right]^2 - 0.025 
\end{equation}
and it is enough to show that
\begin{equation}
\frac{1.03}{\D - 2} \leq \D\left[\frac{2.5 \times 3.9}{24(\D - 2)}\right]^2 - 0.025
\end{equation}
We have $\D \geq 2\s_0 + \s_1 + 1 > 2$, so multipling both sides by $\D - 2 > 0$, this amounts to solving a second-degree polynomial inequality.
Numerically, the zeros of the resulting second-degree polynomial are $\D \approx -33$ and $\D \approx 2.37$. The inequality holds at $\D = 2.3$, and so it holds for all $2 < \D \leq 2.37$. In particular, it holds for $2\s_0 + \s_1 + 1 < \D \leq 4\s_0 + 2\s_1$ when $1 \leq 2\s_0 + \s_1 \leq 1.1$.

{\bf Case of $R_3$:}\\
Lastly, consider $R_3$. Here more extensive numerical methods will be used, and we begin by reducing the analysis from three variables to two. Divide $R_3$ into four subregions, 
\begin{align*}
R_{3, 1}& = \{(\s_0, \s_1, \D) \in R_3 : 1 / 2 \leq \s_1 \leq 1\},\\
R_{3, 2}& = \{(\s_0, \s_1, \D) \in R_3 : 1/ 4 \leq \s_1 < 1 / 2\},\\
R_{3, 3}& = \{(\s_0, \s_1, \D) \in R_3 : 1/8 \leq \s_1 < 1 / 4\},\\
R_{3, 4}& = \{(\s_0, \s_1, \D) \in R_3 : 0 \leq \s_1 < 1/8\}.
\end{align*} 
Define 
$$u_1 = 5.5,\quad u_2 = 5.75, \quad u_3 = 5.875, \quad u_4 = 5.9375.$$ 
Then 
$$6 - 3\s_0 - 2\s_1 = \brac{6 - \frac{\s_1}{2}} - 3\s_0 - \frac{3\s_1}{2} \geq u_i - \frac{3(2\s_0 + \s_1)}{2}$$ 
in $R_{3, i}$, $i = 1, 2, 3, 4$. 

Fixing $i$, (\ref{l1l2}) will hold in $R_{3, i}$ if we can show that
\begin{equation}\label{fractionalphi}
 \frac{1 + 3(4\s_0 + 2\s_1 - \D)^2}{\D - 2} \leq \D\left[\frac{\l(u_i-3(2\s_0 + \s_1)/2)}{24(\D - 2\s_0 - \s_1 - 1)}\right]^2 - \frac{(4\s_0 + 2\s_1 - \D)}{4}
\end{equation}
Note that $\s_0, \s_1$ only appear as $\S = 2\s_0 + \s_1$ in (\ref{fractionalphi}). For this reason we clear denominators in (\ref{fractionalphi}) and define for $i = 1, 2, 3, 4,$
\begin{eqnarray*}
\varphi_i(\S, \D) &=& \l^2\D(\D-2)(u_i - 3\S / 2)^2 - 144(\D-2)(\D-\S - 1)^2(2\S - \D) \\
&& - 576(\D - \S - 1)^2 - 1728(\D - \S - 1)^2(2\S - \D)^2.
\end{eqnarray*}
In which case, \eqref{fractionalphi} is equivalent to $\varphi_i(\S,\D)\geq 0$.

In $R_{3, 1}$ we have $1.1 \leq \S \leq 1.5$ since $2\s_0 + \s_1 \geq 1.1$ is assumed, and $\s_1 \geq 1/2$ and $\s_0 + \s_1 \leq 1$ imply $2\s_0 + \s_1 \leq 2 - \s_1 \leq 1.5$. For this reason define 
\begin{align*}
\widetilde{R}_{3,1}& = \{(\S, \D) : 1.1 \leq \S \leq 1.5, \S + 1 \leq \D \leq 2\S\}\\
\widetilde{R}_{3,2}& = \{(\S, \D) : 1.5 \leq \S \leq 1.75, \S + 1 \leq \D \leq 2\S\},\\ \widetilde{R}_{3,3}& = \{(\S, \D) : 1.75\leq\S\leq 1.875, \S + 1 \leq \D \leq \min\{2\S, 3.6\}\},\\ 
\widetilde{R}_{3,4}& = \{(\S, \D) : 1.875\leq\S\leq 2, \S + 1 \leq \D \leq \min\{2\S, 3.6\}\}.
\end{align*}
Here $\S + 1 \leq \D \leq 2\S$ is (\ref{eq:dbound}).

Equation (\ref{fractionalphi}) will follow from showing that $\varphi_i(\S, \D) \geq 0$ in $\widetilde{R}_{3,i}$, $i = 1, 2, 3, 4$.

The $\varphi_i$ are degree four polynomials, and bounds on $|\nabla \varphi_i|$ are found by applying the triangle inequality to the partial derivatives of $\varphi_i$. The same bound will be applied to $\nabla \varphi_i$ for all $i$. using,
$$2 \leq \S + 1 \leq \D \leq 2\S \leq 4,\quad u_i \leq 6,\quad \l < 3$$
from which we obtain
\begin{align*}
&u_i - \frac{3\S}{2} \leq \frac{9}2,\; -1 \leq 3\S -2\D+1 \leq 1, -2 \leq 4\S-3\D+2\leq 1,\\
&(\D-\S-1)(2\S-\D) \leq \frac{(\S-1)^2}4 \leq \frac14.
\end{align*}
we have
\begin{eqnarray*}
\left|\frac{\partial \varphi_i}{\partial \S}\right| &=& |-3\l^2\D(\D-2)(u_i - 3\S /  2) + 288(\D - 2)(\D-\S-1)(2\S-\D) \\
&& - 288(\D - 2)(\D-\S-1)^2 + 1152(\D-\S-1) \\
&& + 3456(\D - \S - 1)(2\S-\D)^2 - 6912(\D-\S-1)^2(2\S-\D)| \\
&\leq& 3\l^2\D(\D-2)(u_i-3\S / 2) + 288(\D-2)(\D-\S-1)|3\S - 2\D + 1| \\
&& + 1152(\D-\S-1) + 3456(\D-\S-1)(2\S-\D)|4\S - 3\D + 2| \\
&\leq& 27 \cdot 4 \cdot 2 \cdot 9 / 2 + 288 \cdot 2 \cdot 2 \cdot 1 + 1152 \cdot 2 + 3456 \cdot 3/4 \cdot 2 \\
&=& 9612
\end{eqnarray*}
 For $\D$,
\begin{eqnarray*}
\left|\frac{\partial \varphi_i}{\partial \D}\right| &=& |\l^2 \D(u_i - 3\S / 2)^2 + \l^2(\D-2)(u_i-3\S/2)^2  - 144(\D-\S-1)^2(2\S-\D) \\
&& - 288(\D-2)(\D-\S-1)(2\S-\D)  + 144(\D-2)(\D-\S-1)^2 \\
&& - 1152(\D-\S-1) - 3456(\D-\S-1)(2\S-\D)^2 + 3456(\D-\S-1)^2(2\S-\D)| \\
&\leq& \l^2(2\D-2)(u_i-3\S/2)^2 + 144(\D-\S-1)^2(2\S-\D) + 288(\D-2)(\D-\S-1)(2\S-\D) \\
&& + 144(\D-2)(\D-\S-1)^2 + 1152(\D-\S-1) + 3456(\D-\S-1)(2\S-\D)|2\D - 3\S - 1| \\
&\leq& 9\cdot 6 \cdot (9/2)^2 + 144 \cdot 2^2 \cdot 1  + 288\cdot 2 \cdot 2 \cdot 1 + 144 \cdot 2 \cdot 2^2 + 1152 \cdot 2 + 3456\cdot 3/4 \cdot 1 \\
&=& 8383.5
\end{eqnarray*}
so $|\nabla \varphi_i| \leq 12755$ for $i = 1, 2, 3, 4$.

For each $i$, a grid $\mathcal{P}_i \subseteq \widetilde{R}_{3, i}$ of $4 \cdot 10^6$ points is generated such that for each $x \in \widetilde{R}_{3, i}$, there exists an $x_0 \in \mathcal{P}_i$ for which $|x - x_0| \leq 0.001$. On this grid, $\varphi_i$ is calculated numerically, and it is found that
\begin{equation}
\min_{x_0 \in \mathcal{P}_i} \varphi_i(x_0) = \left\{\begin{array}{ll} 
22.49, & i = 1 \\
25.50, & i = 2 \\
27.08, & i = 3 \\
19.04, & i = 4
\end{array} \right.
\end{equation}
So for any $i$ and any $x \in \widetilde{R}_{3, i}$, there exists an $x_0$ such that $|\varphi_i(x) - \varphi_i(x_0)| \leq |\nabla \varphi_i| |x - x_0| \leq 12755 \cdot 0.001 < 13$, which implies $\varphi_i(x) > \varphi_i(x_0) -13 > 0$. This proves (\ref{l1l2}) for $\s_0, \s_1, \D \in R_3$.

\end{proof}

\section{Appendix C} 

This section is concerned with the functions 
$$f_0(x) = e^x\text{ and }f_k(x) = e^x - \sum_{j=0}^{k-1} \frac{x^j}{j!},\  x \geq 0,\ k = 1, 2, 3,$$ 
and the related functions
\begin{equation}
g_0(x) = \frac{x f_2(x)}{f_3(x)}, \quad g_1(x) = \frac{x f_1(x)}{f_2(x)}, \quad g_2(x) = \frac{x f_0(x)}{f_1(x)}.
\end{equation}
Since $f_k(0) = 0$ for $k \geq 1$, we define $g_i(0) = \lim_{x \rightarrow 0} g_i(x) = 3 - i$. Note that
\begin{equation}
\frac{d}{dx} f_k(x) = f_{k-1}(x), \quad k \geq 1
\end{equation}

\begin{lemma}\label{gibounds}
For all $x \geq 0$ and $i = 0, 1, 2$,
\begin{equation}
x < g_i(x) \leq 3 - i + x
\end{equation}
with equality in the upper bound if and only if $x = 0$.
\end{lemma}

\begin{proof}
Fix $i$. By definition, $g_i(0) = 3 - i$. For $x > 0$ consider
\begin{equation}
g_i(x) - x = \frac{x f_{2 - i}(x)}{f_{3 - i}(x)} - x = \frac{x(f_{2-i}(x) - f_{3-i}(x))}{f_{3-i}(x)} = \frac{x^{3 - i}}{(2-i)! f_{3-i}(x)}.
\end{equation}
Since $f_{3 - i}(x) > 0$ we have $g_i(x) - x > 0$. Now
\begin{equation}
(3 - i)(2-i)! f_{3-i}(x) - x^{3 - i} = (3 - i)!\sum_{k \geq 3 - i} \frac{x^k}{k!} - x^{3 - i} = (3 - i)! \sum_{k \geq 4 - i} \frac{x^k}{k!} > 0
\end{equation}
for $x > 0$, implying $g_i(x) - x < 3 - i$.
\end{proof}

\begin{lemma}\label{gconvex}
The functions $g_0, g_1, g_2$ are convex, and $g_i'(x) \geq 1 / (4-i)$ for $x \geq 0$, $i = 0, 1, 2$.
\end{lemma}
\begin{proof}
Consider $g_0$. Since $f_2(x) = f_3(x) + x^2 / 2$, $g_0$ can be written as
\begin{equation}
g_0(x) = \frac{xf_2(x)}{f_3(x)} = x + \frac{x^3}{2f_3(x)}
\end{equation}
Let $q(x) = f_3(x) / x^3 = \sum_{j \geq 0} x^j / (j + 3)!$. Then $g_0(x) = x + 1 / 2q(x)$, and
\begin{equation}
g_0'(x) = 1 - \frac{q'(x)}{2q(x)^2}, \quad g_0''(x) = \frac{2q'(x)^2 - q(x)q''(x)}{2q(x)^3}
\end{equation}
and we show that $2q'(x)^2 - q(x)q''(x) \geq 0$. We have $q'(x) = \sum_{j \geq 0} (j+1)x^j / (j + 4)!$ and $q''(x) = \sum_{j \geq 0} (j + 1)(j+2) x^j / (j + 5)!$, so the $j$th Taylor coefficient of $2q'(x)^2 - q(x)q''(x)$ is given by
\begin{eqnarray*}
[x^j] [ 2q'(x)^2 - q(x)q''(x)] &=& \sum_{\stackrel{j_1, j_2 \geq 0}{j_1 + j_2 = j}} 2\frac{(j_1+1)}{(j_1+4)!}\frac{(j_2+1)}{(j_2+4)!} - \frac{1}{(j_1 + 3)!}\frac{(j_2 + 1)(j_2 + 2)}{(j_2 + 5)!} \\
&=& \sum_{j_1, j_2} \frac{2(j_1 + 1)(j_2 + 1)(j_2 + 5) - (j_1 + 4)(j_2 + 1)(j_2 + 2)}{(j_1 + 4)!(j_2 + 5)!} \\
&=& \sum_{j_1, j_2} \frac{(j_2 + 1)(2(j_1 + 1)(j_2 + 5) - (j_1 + 4)(j_2 + 2))}{(j_1 + 4)!(j_2 + 5)!} \\
&=& \sum_{j_1, j_2} \frac{(j_2 + 1)(j_1j_2 + 8j_1 - 2j_2 + 2)}{(j_1 + 4)!(j_2 + 5)!}
\end{eqnarray*}
It is seen that this is positive for $j = 0, 1, 2$. Let $Q(j_1, j_2)$ denote the summand. If $j \geq 3$ then since $Q(j_1, j_2) \geq 0$ whenever $j_1 \geq 2$.
\begin{eqnarray*}
\sum_{\stackrel{j_1, j_2 \geq 0}{j_1 + j_2 = j}} Q(j_1, j_2) &\geq& Q\left(\left\lfloor \frac{j}{2} \right\rfloor, \left\lceil \frac{j}{2}\right\rceil \right) + Q(0, j) + Q(1, j-1) \\
&=& \frac{\left(\rdown{j/2} + 1\right)\left(\rdup{j/2}\rdown{j/2} + 8\rdup{j/2}-2\rdown{j/2}+2\right)}{(\rdup{j/2}+4)!(\rdown{j/2}+5)!} - \frac{2(j^2-1)}{24(j+5)!} - \frac{j^2 - 11j}{120(j+4)!} \\
&\geq& \frac{j^3}{8(\rdup{j/2} + 4)!(\rdown{j/2}+5)!} - \frac{j^2}{12(j+5)!} - \frac{j^2-11j}{120(j+4)!} \\
&=& \frac{j^3}{8(\rdup{j/2} + 4)!(\rdown{j/2}+5)!} - \frac{10j^2 + (j^2-11j)(j+5)}{120(j+5)!} \\
&\geq& \frac{j^3}{8}\left(\frac{1}{(\rdup{j/2} + 4)!(\rdown{j/2}+5)!} - \frac{1}{15(j+5)!}\right).
\end{eqnarray*}
(To get the final inequality, consider $j\leq 11$ and $j>11$ seperately).

It remains to show that $a_j = (\rdup{j/2}+4)!(\rdown{j/2}+5)!$ is smaller than $b_j = 15(j + 5)!$ for $j \geq 3$. For $j = 3$, $a_3 = 6! \cdot 6! < 15 \cdot 8! = b_3$. For the induction step, $a_{j+1} / a_j \leq j/2 + 6$ while $b_{j+1} / b_j = j+6$, so $a_3 < b_3$ implies $a_j < b_j$ for all $j \geq 3$. So $2q'(x)^2 - q(x)q''(x) \geq 0$, and it follows that $g_0$ is convex. Similar arguments show that $g_1, g_2$ are convex.

For $i = 0, 1$,
\begin{align*}
g_i'(x)&=\frac{f_{2-i}(x)}{f_{3-i}(x)}+\frac{xf_{1-i}(x)}{f_{3-i}(x)}-\frac{xf_{2-i}(x)^2}{f_{3-i}(x)^2}\\
&=\frac{f_{2-i}(x)f_{3-i}(x)+xf_{1-i}(x)f_{3-i}(x)-xf_{3-i}(x)^2}{f_{3-i}(x)^2}.
\end{align*}
Now
\begin{align*}
&f_{2-i}(x)f_{3-i}(x)+xf_{1-i}(x)f_{3-i}(x)-xf_{3-i}(x)^2=\\
&x^{6-2i}\brac{\frac{1}{(2-i)!(4-i)!}+ \frac{1}{(3-i)!^2}+\frac{1}{(1-i)!(4-i)!}+\frac{1}{(2-i)!(3-i)!}-\frac{2}{(2-i)!(3-i)!}+O(x)}\\
&=x^{6-2i}\brac{\frac{1}{(3-i)!(4-i)!}+O(x)}.
\end{align*}
And
$$f_{3-i}(x)^2=x^{6-2i}\brac{\frac{1}{(3-i)!^2}+O(x)}.$$
So, for $i=0,1$ we have
$$g_i'(x)=\frac{1}{4-i}+O(x).$$
For $i=2$ we have
$$g_2'(x)=\frac{e^x}{f_1(x)}+\frac{xe^x}{f_1(x)}-\frac{xe^{2x}}{f_1(x)^2} =e^x\bfrac{f_1(x)(1+x)-xe^x}{f_1(x)^2} =e^x\bfrac{\frac{x^2}{2}+O(x^3)}{x^2+O(x^3)}=\frac12+O(x).$$
And by the convexity of $g_i$ we have $g_i'(x) \geq 1 / (4-i)$ for all $x \geq 0$.
\end{proof}

Lemma \ref{gconvex} allows us to define inverses $g_i^{-1}$, $i=0,1,2$.
\begin{lemma} \label{finequalities}
For $0 \leq x \leq \l = g_0^{-1}(4)$, the following inequalities hold.
\begin{eqnarray*}
(i) && 1 \leq\frac{f_2(x)^2}{f_1(x)f_3(x)} \leq 2\\
(ii) && 0.09 < \frac{f_3(x)}{x^2f_1(x)} \leq \frac{1}{6} \\
(iii) && \frac{f_2(x)}{x f_1(x)} \leq \frac{1}{3} \\
(iv) && 0.01 < \frac{f_3(x)}{x^4 f_1(x)}  \\
(v) && 0.09 < \frac{f_3(x)}{x^2 f_2(x)}
\end{eqnarray*}
\end{lemma}
\begin{proof}
Consider (i). For the lower bound, let $x > 0$ and consider the equation $f_2(x)^2 = f_1(x) f_3(x)$. By definition of $f_i$, this equation can be written as
\begin{equation}
(e^x - 1 - x)^2  = (e^x - 1)\left(e^x - 1 - x - \frac{x^2}{2}\right)
\end{equation}
Expanding and reordering terms, we have
\begin{equation}
e^x\left(x + \frac{x^2}{2}\right) = x + \frac{x^2}{2}
\end{equation}
which clearly has no positive solution. Since $f_2(0)^2 / f_1(0)f_3(0) = 3/2 >1$, this implies that $f_2(x)^2 / f_1(x) f_3(x) > 1$ for all $x \geq 0$.

For the upper bound we consider the equation $f_2(x)^2 = 2f_1(x) f_3(x)$. This simplifies to 
$$(e^x-1)^2=x^2e^x\text{ or }e^x=1+xe^{x/2}$$
which has no positive solution.

Since $g_0, g_1$ are increasing by Lemma \ref{gconvex} and positive, the expressions in (ii) -- (v) are all decreasing;
\begin{equation}
\frac{f_3(x)}{x^2f_1(x)} = \frac{1}{g_0(x)g_1(x)}, \quad \frac{f_2(x)}{xf_1(x)} = \frac{1}{g_1(x)}, \quad \frac{f_3(x)}{x^4f_1(x)} = \frac{1}{x^2g_0(x)g_1(x) }, \quad \frac{f_3(x)}{x^2f_2(x)} = \frac{1}{xg_0(x)}
\end{equation}
The upper bounds are obtained by noting that $g_i(0) = 3 - i$ by Lemma \ref{gibounds}, while the lower bounds are obtained numerically by letting $x = 2.688 > \l$.
\end{proof}

\end{document}